\documentclass[11pt,a4paper,reqno]{amsart}

\calclayout

\usepackage[utf8]{inputenc}
\usepackage[T1]{fontenc}

\usepackage{appendix}
\usepackage{import} 

\usepackage{amsmath, amsthm, amsfonts, amssymb}
\usepackage{mathtools}
\usepackage{mathrsfs} 
\usepackage{bbm} 
\usepackage{bm} 

\numberwithin{equation}{section} 
\numberwithin{table}{section} 
\numberwithin{figure}{section} 

\newdimen\oldparindent
\usepackage{parskip} 
\makeatletter 
\def\thm@space@setup{ \thm@preskip=\parskip \thm@postskip=0pt }
\makeatother

\usepackage[normalem]{ulem} 
\usepackage{courier} 
\usepackage[table,xcdraw]{xcolor} 

\usepackage[shortlabels]{enumitem} 
\setlist[enumerate, 1]{label={(\roman*)}} 

\usepackage[labelfont=bf]{caption} 
\usepackage{subcaption} 
\usepackage{graphicx} 
\usepackage{afterpage}
\usepackage{wrapfig} 
\usepackage[export]{adjustbox} 
\usepackage{float} 
\usepackage{makecell} 
\usepackage{booktabs} 
\usepackage{tabularx} 
\usepackage{multirow} 

\usepackage[backend=biber, style=alphabetic, sorting=nyt, maxbibnames=6, giveninits=true]{biblatex}

\renewbibmacro{in:}{%
  \ifboolexpr{%
     test {\ifentrytype{article}}%
     or
     test {\ifentrytype{inproceedings}}%
  }{}{\printtext{\bibstring{in}\intitlepunct}}%
}

\usepackage{hyperref} 
\hypersetup{
  colorlinks   = true, 
  urlcolor     = blue, 
  linkcolor    = blue, 
  citecolor    = red   
}
\usepackage{url}
\theoremstyle{plain} 
\newtheorem{theorem}{Theorem}[section] 
\newtheorem{lemma}{Lemma}[section]
\newtheorem{proposition}{Proposition}[section]
\newtheorem{condition}{Condition}

\theoremstyle{definition} 
\newtheorem{remark}{Remark}[section]
\newtheorem*{remark*}{Remark}

\newtheorem*{definition*}{Definition}
\newtheorem*{theorem*}{Theorem} 
\newtheorem*{corollary*}{Corollary}
\newtheorem*{lemma*}{Lemma}
\newtheorem*{proposition*}{Proposition}

\newtheorem*{note*}{Note}

\newtheorem*{example*}{Example}

\newtheorem*{exercise*}{Exercise}

\newcommand{\dd}{\;\mathrm{d}} 
\newcommand{\Prob}[2][]{\mathbb{P}_{#1}\left(#2\right)} 
\newcommand{\E}{\mathbb{E}} 
\newcommand{\ev}[2][]{\mathbb{E}_{#1}\left[#2\right]} 
\newcommand{\cev}[3][]{\mathbb{E}_{#1}\left[#2 \mid #3 \right]} 
\newcommand{\indicator}[1]{\mathbbm{1}_{#1}} 

\newcommand{\reals}{\mathbb{R}} 
\newcommand{\naturals}{\mathbb{N}} 

\newcommand{\norm}[1]{\lVert #1 \rVert} 

\DeclareMathOperator{\dtv}{d_{TV}}

\DeclareMathOperator*{\argmin}{arg\,min}

\let\epsilon\varepsilon

\let\phi\varphi

\title[On approximating the Potts model]{On approximating the Potts model with contracting Glauber dynamics}

\author{Roxanne He}
\address{School of Mathematics and Statistics, The University of Melbourne, Parkville, VIC, 3010, Australia}
\email{roxanne\_he@outlook.com}
\author{Jackie Lok}
\address{Department of Operations Research and Financial Engineering, Princeton University, Princeton, NJ 08544, USA}
\email{jackie.lok@princeton.edu}

\date{25 September 2025}

\begin{document}

\begin{abstract}
We show that the Potts model on a graph can be approximated by a sequence of independent and identically distributed spins in terms of Wasserstein distance at high temperatures. We prove a similar result for the Curie--Weiss--Potts model on the complete graph, conditioned on being close enough to any of its equilibrium macrostates, in the low-temperature regime. Our proof technique is based on Stein's method for comparing the stationary distributions of two Glauber dynamics with similar updates, one of which is rapid mixing and contracting on a subset of the state space. Along the way, we prove a new upper bound on the mixing time of the Glauber dynamics for the conditional measure of the Curie--Weiss--Potts model near an equilibrium macrostate.
\end{abstract}



\maketitle


\section{Introduction} \label{sec:intro}

The Potts model is a spin system that generalises the classical Ising model of magnetism, and has been extensively studied in many fields, including statistical physics~\cite{wu1982potts} and probability theory~\cite{fortuin1972random, swendsen1987nonuniversal, edwards1988jointpotts, grimmett2006random}. The Potts model and its extensions have also found applications in areas such as simulating biological cells~\cite{graner1992simulation}, predicting protein structure~\cite{sercu2021neural}, image reconstruction~\cite{geman1984image}, and community detection in complex networks~\cite{reichardt2004community}.

For a graph $G = (V, E)$ on $N$ vertices, a \emph{configuration} or \emph{colouring} $\sigma \in [q]^V$ is a function which assigns to each vertex $v \in V$ a \emph{spin} or \emph{colour} $\sigma(v) \in [q] := \{ 1, \ldots, q \}$. Under the ferromagnetic Potts model, the probability of each configuration is given by the \emph{Gibbs measure} $\mu$ with
\[
    \mu(\sigma) = \frac{e^{-\beta H(\sigma)}}{Z(\beta)},
\]
where $\beta \geq 0$ is an \emph{inverse temperature parameter},
\[
    H(\sigma) = -\frac{1}{N} \sum_{u, v \in V} \indicator{\sigma(u) = \sigma(v)}
\]
is the Hamiltonian with \emph{interaction strength} $N^{-1}$, and $Z(\beta) = \sum_{\sigma} e^{-\beta H(\sigma)}$ is the \emph{partition function} (i.e.\ normalisation factor), which is difficult to compute. Thus, configurations with more monochromatic edges (i.e.\ whose endpoints have the same colour) are more likely. We will assume that $q \geq 3$ (note that $q = 2$ corresponds to the Ising model). 

A high-level heuristic for the Potts model is that the spins should be approximately independent if the temperature is high (i.e.\ $\beta$ is small). 
In this paper, we prove that the Potts model on a general bounded-degree graph is close to a random configuration with independent and uniformly distributed spins in terms of Wasserstein distance when $\beta$ is small enough (Theorem~\ref{thm:potts_approx_result}).
We prove a similar approximation result for the Curie--Weiss--Potts model on the complete graph for a wider range of $\beta$ in a high-temperature regime where the corresponding Glauber dynamics is known to mix rapidly (Theorem~\ref{WDbound_high}).
Furthermore, we show that in the complementary low-temperature regime, the Curie--Weiss--Potts model, conditioned on being close to any of its equilibrium macrostates, can be approximated by a sequence of i.i.d.\ spins (Theorem~\ref{WDbound}).
Along the way, we prove a new upper bound on the mixing time of the Glauber dynamics for the conditional measure of the Curie--Weiss--Potts model near an equilibrium macrostate (Theorem~\ref{mixingtime}).

Our main tool is the use of Stein's method to reduce the problem of comparing two distributions to the problem of comparing the dynamics of two Markov chains that these distributions are stationary for. This idea was introduced by the concurrent works~\cite{reinert2019approximating} and~\cite{bresler2019stein}, where it was used to approximate exponential random graphs (with Erd\H{o}s-R\'{e}nyi random graphs), and the Ising model on $d$-regular expander graphs (with the Curie-Weiss model) respectively. This technique was further applied in~\cite{blanca2022mixing} to establish spectral independence of spin systems.

To give a more concrete statement, the following approximation result was obtained in~\cite{blanca2022mixing}, generalising the ideas introduced in~\cite{reinert2019approximating, bresler2019stein}.
We say that a Markov chain on a metric space $(\Sigma, d)$ is \emph{contracting} if for all $\sigma, \tau \in \Sigma$, there exists a coupling $(X^\sigma_1, X^\tau_1)$ of the one-step distributions of the chain, starting from $\sigma$ and $\tau$, such that for some $0 \leq \kappa < 1$,
\begin{equation} \label{eq:defn_contracting}
    \ev{d(X^\sigma_1, X^\tau_1)} \leq \kappa \cdot d(\sigma, \tau).
\end{equation}

\begin{theorem}[{\cite[Lemma 4.3]{blanca2022mixing}}] \label{thm:approx_thm_general_contracting}
Let $P$ and $Q$ be Markov chains on a finite metric space $(\Sigma, d)$ with stationary distributions $\mu$ and $\nu$ respectively. Denote the one-step distributions of $P$ and $Q$, starting from $\sigma \in \Sigma$, by $P(\sigma, \cdot)$ and $Q(\sigma, \cdot)$. Let $X \sim \mu$ and $Y \sim \nu$ be random vectors. If $P$ is contracting according to~\eqref{eq:defn_contracting} for some $0 \leq \kappa < 1$, then for any function $h: \Sigma \to \mathbb{R}$,
\[
    |\mathbb{E} h(X) - \mathbb{E} h(Y)| \leq \frac{L_d(h)}{1 - \kappa} \ev{d_W(P(Y, \cdot), Q(Y, \cdot))} ,
\]
where $d_W$ is the Wasserstein distance between measures on $\Sigma$ with respect to $d$, and $L_d(h)$ is the optimal Lipschitz constant of $h$ such that $|h(\sigma) - h(\tau)| \leq L_d(h) \cdot d(\sigma, \tau)$ for all $\sigma, \tau \in \Sigma$.
\end{theorem}

\begin{remark} \label{rmk:wasserstein_interp}
It is well known that the Wasserstein distance between two measures $\mu$ and $\nu$ on a finite metric space $(\Sigma, d)$ is given by the following two equivalent dual representations~\cite{villani2003optimal}:
\begin{equation} \label{eq:wasserstein_defn}
    d_W(\mu, \nu) = \sup_{h} \{ \E h(X) - \E h(Y) : X \sim \mu, Y \sim \nu \} = \inf_{(X, Y)} \ev{d(X, Y)}.
\end{equation}
Here, the supremum is over all $1$-Lipschitz functions $h: \Sigma \to \reals$ with respect to $d$, and the infimum is over all couplings $(X, Y)$ of $\mu$ and $\nu$ (i.e.\ joint distributions on $\Sigma \times \Sigma$ such that $X \sim \mu$ and $Y \sim \nu$), which is attained by an optimal coupling.
Therefore, by considering the class of $1$-Lipschitz functions $h$, the approximation result in Theorem~\ref{thm:approx_thm_general_contracting} also implies the following bound on the Wasserstein distance:
\begin{equation}
    d_W(\mu, \nu) \leq \frac{1}{1 - \kappa} \ev{d_W(P(Y, \cdot), Q(Y, \cdot))}.
\end{equation}
Assuming that $P$ is contracting, this provides a precise statement of the heuristic that if the two chains $P$ and $Q$ have similar updates (i.e.\ $\ev{d_W(P(Y, \cdot), Q(Y, \cdot))}$ is small), then their stationary distributions $\mu$ and $\nu$ are also close (in terms of Wasserstein distance).
\end{remark}

In our setting, we will compare the distributions of two spin systems $\mu$ and $\nu$ on the state space $\Sigma = [q]^V$ equipped with the Hamming distance $d_H(\sigma, \tau) = \sum_{v \in V} \indicator{\sigma(v) \ne \tau(v)}$, which counts the number of vertices with different colours.
We will choose $P$ to be the \emph{Glauber dynamics} for $\mu$, which is a discrete-time Markov chain $(\sigma_t)_{t \geq 0}$ with $\mu$ as its stationary distribution, and the following transitions: given the current configuration $\sigma_t$, a vertex $v \in V$ is chosen uniformly at random, and a new configuration $\sigma_{t+1}$ is generated by recolouring $v$ with a new colour $k \in [q]$ drawn according to $\mu$, conditional on the colours of all the other vertices being fixed. That is, $\sigma_{t+1}(u) = \sigma_t(u)$ for all $u \ne v$, and $\sigma_{t+1}(v) = k$ with probability
\begin{equation} \label{eq:cond_spin_dist}
    \mu_v(k \mid \sigma_t) := \mu(\sigma(v) = k \mid \sigma(w) = \sigma_t(w) \;\; \forall w \neq v), \quad k \in [q] .
\end{equation}
We say that $\mu_v(\cdot \mid \sigma)$ is the \emph{conditional spin distribution} of $\mu$ at $v$, given $\sigma$.
Similarly, we will choose $Q$ to be the Glauber dynamics for $\nu$.
We will take $\mu$ to be the more complicated model of interest (i.e.\ the Potts model) and $\nu$ to be a simpler model (i.e.\ with i.i.d.\ spins), chosen in a specific way such that the transition probabilities of $P$ and $Q$ are ``matched'' using a mean-field approximation (see~\eqref{eq:potts_meanfield} later).

However, a technical challenge often encountered in practice is that the Glauber dynamics $P$ is only \emph{contracting on a subset of the state space}; see, e.g., the exponential random graph model analysed in~\cite{reinert2019approximating}, and our upcoming discussion of the Curie--Weiss--Potts model. 
We demonstrate that this problem can be overcome to deduce similar approximation results as Theorem~\ref{thm:approx_thm_general_contracting}, provided that the following high-level conditions can be shown to hold:
\begin{enumerate}[label=(\arabic*)]
    \item The chain $P$ is \emph{rapid mixing}: that is, its \emph{mixing time}
    \begin{equation} \label{eq:defn_mixingtime}
        t_{\mathrm{mix}}(\epsilon) := \inf \left\{ t \geq 0: \max_{\sigma \in [q]^V} \norm{P^t(\sigma, \cdot) - \mu}_{\mathrm{TV}} \leq \epsilon \right\},
    \end{equation}
    which measures the time required for the total variation distance between the $t$-step distribution of the chain (in the worst case over all initial states $\sigma$), denoted by $P^t(\sigma, \cdot)$, and its stationary distribution $\mu$ to fall below a given threshold $\epsilon < 1/2$, 
    can be upper bounded by a polynomial in $N$, the number of vertices.
    
    \item The chain $P$ is contracting in some subset $\widetilde{\Sigma} \subseteq \Sigma$ of the state space. Furthermore, starting in another subset $\widetilde{\Sigma}_0 \subseteq \widetilde{\Sigma}$, the chain $P$ remains in $\widetilde{\Sigma}$ for a sufficiently long period relative to its mixing time \emph{with high probability} (i.e.\ the event does not occur with probability exponentially small in $N$).
    
    \item The random vector $Y \in \widetilde{\Sigma}_0$ with high probability.
\end{enumerate}
These ideas are essentially embedded in the proof of~\cite[Theorem~1.13]{reinert2019approximating}.
Our main contribution is in making this strategy explicit, and using it to prove our approximation results for the Curie--Weiss--Potts model in Section~\ref{CWP}, for which the Glauber dynamics is not contracting when $\beta$ is large.
In particular, Theorems~\ref{WDbound} and~\ref{mixingtime} are the most technically demanding results proved in the low-temperature regime \emph{where the Glauber dynamics is not rapidly mixing and there are multiple equilibrium macrostates}.
As a byproduct, we obtain a new upper bound on the mixing time of the Glauber dynamics for the Curie--Weiss--Potts model, conditioned on being close to any of its equilibrium macrostates.
This may be of independent interest, since the rapid mixing of Glauber dynamics for conditional distributions is less well-understood.

Before stating our results, we define some notation that will be used. For any function $h: [q]^V \to \reals$, let
\begin{equation} \label{eq:componentwise_lipschitz}
    L_v(h) := \sup_{\sigma(u) = \tau(u) \,\forall u \ne v} |h(\sigma) - h(\tau)|
\end{equation}
be the Lipschitz constant of $h$ in the component corresponding to $v \in V$ with respect to $d_H$, where the supremum is taken over all pairs $\sigma, \tau \in [q]^V$ that only possibly differ at $v$, and denote the associated vector by $L(h) := (L_v(h))_{v \in V} \in \reals^V$.
Thus, the usual vector $\ell_\infty$ norm $\norm{L(h)}_\infty$ denotes the maximum value of $L_v(h)$ for any $v \in V$.
Observe that $\norm{L(h)}_\infty$ is equal to the optimal Lipschitz constant of $h$ with respect to the Hamming distance $d_H$.

We also use standard asymptotic notation as $N \to \infty$ (treating $q \geq 3$ and $\beta > 0$ as constants): we write $f(N) = O(g(N))$ if there exists an absolute constant $C > 0$ such that $|f(N)| \leq C |g(N)|$ for sufficiently large $N$, and $f(N) = o(1)$ if $|f(N)| \to 0$.

\subsection{Potts model on bounded-degree graphs}

As a more straightforward example of the kinds of results that we are aiming for, we first state an approximation result for the Potts model on a general graph that can be obtained from Theorem~\ref{thm:approx_thm_general_contracting}. We show that if $\beta$ is small enough such that the Glauber dynamics is contracting (and hence rapid mixing), then the Potts model is close to a random configuration with independent, uniformly distributed spins.

\begin{theorem} \label{thm:potts_approx_result}
Let $G = (V, E)$ be a graph on $N$ vertices with maximum degree $\Delta$ and $|E|$ edges. Let $X \in [q]^V$ be distributed according to the Potts model on $G$ with inverse temperature $\beta$, and $Y \in [q]^V$ be a random configuration where the colour of each vertex is sampled independently and uniformly at random. If $\Delta \tanh(\beta / N) < 1$, then for any function $h: [q]^V \to \reals$,
\[
    | \mathbb{E} h(X) - \mathbb{E} h(Y) |
    \leq \norm{L(h)}_\infty \frac{\beta \sqrt{q-1}}{1 - \Delta \tanh(\beta / N)} \sqrt{\frac{2|E|}{N}}.
\]
\end{theorem}

In particular, Theorem~\ref{thm:potts_approx_result} holds if $\beta < N / \Delta$, since $\tanh x \leq x$ for $x \geq 0$.
To interpret Theorem~\ref{thm:potts_approx_result}, observe that the bound tends to zero as $\beta \to 0$ (i.e.\ at infinite temperatures). 
Furthermore, the bound improves for sparser graphs with fewer edges---intuitively, because there are fewer interactions between the vertices---and may be simplified by bounding the average degree $2 |E| / N$ (from the handshaking lemma) by the maximum degree $\Delta$.
In the case when $|E| = O(N^2)$ and $\beta < 1$, the bound implies that the Wasserstein distance between the laws of $X$ and $Y$ is of order $O(\sqrt{N})$ (see Remark~\ref{rmk:wasserstein_interp}). Therefore, the Potts model on a dense graph can be coupled with a sequence of i.i.d.\ spins such that, on average, $O(\sqrt{N})$ of the vertices disagree, which is a vanishingly small proportion of the total number of vertices.
Finally, we note that Theorem~\ref{thm:potts_approx_result} generalises a similar result for the Ising model given in~\cite[Equation~(1.7)]{reinert2019approximating}.

The proof of Theorem~\ref{thm:potts_approx_result} appears in Section~\ref{sec:potts_bounded_degree}, and is relatively straightforward using the fact that the Glauber dynamics for the Potts model is contracting with respect to the Hamming distance according to~\eqref{eq:defn_contracting} whenever $\Delta \tanh(\beta / N) < 1$.

\subsection{Curie--Weiss--Potts model} \label{sec:intro_cwp}

We now turn our focus to the \emph{Curie--Weiss--Potts} (or mean-field Potts) model, which is the Potts model defined on the complete graph. The approximation results in this section, separated into the high-temperature and low-temperature regimes, describe the Curie--Weiss--Potts model at all inverse temperatures $\beta$, and do not follow from Theorem~\ref{thm:approx_thm_general_contracting} since the Glauber dynamics is not globally contracting for large $\beta$.
Here, we will identify the vertices with $[N] = \{1, \ldots, N\}$, and denote the state space by $\Omega := [q]^{N}$.
Let $S: \Omega \to \mathcal{S}$ be the map which sends any configuration $\sigma \in \Omega$ to the \emph{vector of proportions}
\begin{equation} \label{eq:CWP_vector_of_proportions}
    S(\sigma) := (S^{(1)}(\sigma), \ldots, S^{(q)}(\sigma))
\end{equation}
in the probability simplex $\mathcal{S} := \{ x \in \mathbb{R}_+^q: \lVert x \rVert_1 = 1\}$, where
\begin{equation*}
    S^{(k)}(\sigma) := \frac{1}{N} \sum_{j=1}^N \indicator{ \sigma(j)=k }, \quad k \in [q] .
\end{equation*}

In the absence of geometry, the state of the system is effectively characterised by the vector of proportions.
It is well known~\cite{ellis1990limit, costeniuc2005complete} that there exists a \emph{critical inverse temperature}
\begin{equation}
    \beta_c \equiv \beta_c(q) := \frac{(q-1)\log(q-1)}{q-2}
\end{equation}
separating the disordered and ordered phases of the Curie--Weiss--Potts model. The (canonical) \emph{equilibrium macrostates}, which describe equilibrium configurations in the thermodynamic limit, are the global minimisers of the function
\begin{equation} \label{eq:G_beta}
    G_{\beta}(s) := \beta\norm{s}_2^2 - \log\left( \sum_{i=1}^q e^{-2\beta s^i} \right), \quad s \in \reals^q,
\end{equation}
which appear in the Gibbs free energy $\varphi(\beta)$, defined by $2 \beta \varphi(\beta) = \min_{s \in \mathbb{R}^q} G_{\beta}(s) + \log q$. Let
\begin{equation} \label{eq:G_beta_minimisers}
    \mathfrak{S}_{\beta, q} := \argmin_{s \in \reals^q} G_\beta(s)
\end{equation}
be the set of global minimisers of $G_{\beta}$. When $\beta < \beta_c$, there is a unique equilibrium macrostate
\[
    \hat{e} := (1/q, \ldots, 1/q),
\]
corresponding to the disordered phase in which the most likely configurations have roughly equal proportions of each colour. When $\beta > \beta_c$, there are $q$ equilibrium macrostates, corresponding to the $q$ ordered phases in which the Gibbs measure is supported almost entirely on configurations with a particular dominant colour. At criticality, $\beta = \beta_c$, there are $q + 1$ minima in $\mathfrak{S}_{\beta_c, q}$, reflecting the coexistence of the ordered and disordered phases. Due to the symmetry of $G_{\beta}$, its minimisers are in the probability simplex, and therefore define probability distributions. We defer the precise expressions for the points in $\mathfrak{S}_{\beta, q}$ when $\beta \geq \beta_c$ to Theorem~\ref{pointofconcentration} later. 

Furthermore, a complete analysis of the mixing time of the Glauber dynamics for the Curie--Weiss--Potts model is provided in~\cite{cuff2012}. It is shown that the \emph{spinodal inverse temperature}
\begin{equation}
    \beta_s \equiv \beta_s(q) := \sup\left\{ \beta\geq 0 : \left( 1 + (q-1) e^{2\beta \frac{1 - qx}{q-1}} \right)^{-1} - x \neq 0 \quad\forall x \in (1/q, 1) \right\}
\end{equation}
is a \emph{dynamical threshold}: the mixing time is of order $O(N \log(N))$ with cutoff when $\beta < \beta_s$; of order $O(N^{4/3})$ when $\beta = \beta_s$; and is exponentially large in $N$ when $\beta > \beta_s$. The critical slowdown at $\beta_s$ marks the onset of \emph{metastability}, corresponding to the emergence of local minimisers of the free energy (i.e.\ of $G_\beta$). The coexistence of phases (possibly as metastable states) implies slow mixing since the dynamics must pass through states which are exponentially unlikely. The critical inverse temperatures satisfy $1 < \beta_s < \beta_c < q/2$ for $q \geq 3$.

\subsubsection*{High-temperature regime.}
Theorem~\ref{thm:potts_approx_result} can be applied to the complete graph to deduce that for $\beta < 1$ and any $h: \Omega \to \reals$, $|\E h(X) - \E h(Y)| = O(\sqrt{N})$, where $X$ is distributed according to the Curie--Weiss--Potts model and $Y$ is a random configuration with i.i.d.\ uniform spins.
Since the Glauber dynamics for the Curie--Weiss--Potts model mixes rapidly for all $\beta < \beta_s$ (with $\beta_s > 1$), one might expect that a similar bound should hold in the entire high-temperature regime, based on the heuristic that stationary distributions of rapid mixing Markov chains should be approximately independent (see~\cite{reinert2019approximating} for further discussion).
Indeed, we prove the following:

\begin{theorem} \label{WDbound_high}
Suppose that $\beta < \beta_s$. Let $X \in \Omega$ be distributed according to the Curie--Weiss--Potts model with inverse temperature $\beta$, and $Y \in \Omega$ be a random configuration where the colour of each vertex is sampled independently and uniformly at random. Then there exists a constant $\theta^* > 0$, depending on $\beta$ and $q$, such that for any function $h: \Omega \to \reals$,
\begin{equation*}
    \left\lvert \mathbb{E} h(X) - \mathbb{E} h(Y) \right\rvert \leq \norm{L(h)}_\infty \theta^* \sqrt{N}.
\end{equation*}
\end{theorem}

\begin{remark}[Optimality] \label{rmk:optimality_limit}
By considering the class of $1$-Lipschitz functions $h$, Theorem~\ref{WDbound_high} implies that the Wasserstein distance between $X \sim \mu$ and $Y \sim \nu$ satisfies $d_W(\mu, \nu) = O(\sqrt{N})$ (see Remark~\ref{rmk:wasserstein_interp}). This bound is optimal (in terms of dependence on $N$) based on the matching lower bound furnished by the following argument.
Let $W_X := \sqrt{N} (S(X) - \hat{e})$ and $W_Y := \sqrt{N} (S(Y) - \hat{e})$ be the centred and rescaled vectors of proportions of $X$ and $Y$ respectively. By the usual central limit theorem for the multinomial distribution, $W_Y$ converges weakly as $N \to \infty$ to a multivariate normal vector $N(0, \mathbf{\Sigma}_Y)$, whose covariance matrix $\mathbf{\Sigma}_Y$ has diagonal entries $(q - 1) / q^2$ and off-diagonal entries $-1/q^2$. From~\cite[Theorem~2.4]{ellis1990limit}, it is known that $W_X$ also converges weakly as $N \to \infty$ to a multivariate normal vector $N(0, \mathbf{\Sigma}_X)$, and it can be shown that $\mathbf{\Sigma}_X$ has diagonal entries $(q - 1) / (q^2 - 2q\beta)$ and off-diagonal entries $-1/(q^2 - 2\beta)$ (see the proof of~\cite[Proposition 2.2]{ellis1990limit}). Thus, for all $\beta > 0$, the limiting distributions of $W_X$ and $W_Y$ are different (and they coincide when $\beta \to 0$). 
Observe that for any $1$-Lipschitz function $g: \reals^q \to \reals$ with respect to the $\ell_1$ norm, and any optimal coupling $(X^*, Y^*)$ of $\mu$ and $\nu$ such that $\E d_H(X^*, Y^*) = d_W(\mu, \nu)$,
\begin{equation} \label{limit_comparison}
    \lvert \mathbb{E} g(W_X) - \mathbb{E} g(W_Y) \rvert
    \leq \sqrt{N} \, \mathbb{E}\lVert S(X^*) - S(Y^*) \rVert_1
    \leq 2 N^{-1/2} \, \mathbb{E} d_H(X^*, Y^*),
\end{equation}
since each location where $X^*$ and $Y^*$ differ contributes at most $2N^{-1}$ to $\norm{S(X^*) - S(Y^*)}_1$.
Thus, we must have $\liminf_{N \to \infty} d_W(\mu, \nu)/\sqrt{N} \geq c$ for some constant $c > 0$, otherwise~\eqref{limit_comparison} would imply that $W_X$ and $W_Y$ converge to the same distribution.
\end{remark}

\subsubsection*{Low-temperature regime.}
When $\beta \geq \beta_s$, the Glauber dynamics is not rapid mixing. Moreover, when there exist multiple equilibrium macrostates (when $\beta \geq \beta_c$), it does not make sense to compare the Curie--Weiss--Potts model to a single sequence of i.i.d.\ spins. However, in this low-temperature regime, we will show that the Curie--Weiss--Potts model, conditioned on being close enough to any equilibrium macrostate $x \in \mathfrak{S}_{\beta, q}$, can be approximated by a sequence of i.i.d.\ spins with probabilities given by $x$.
To define the restriction region, let
\begin{equation} \label{def:restriction_region}
    \widetilde{\Omega}(x, r) := \left\{ \sigma \in \Omega : \lVert S(\sigma) - x \rVert_{2} \leq r \right\},
\end{equation}
where $r > 0$ is a constant (only depending on $\beta$ and $q$) that will be chosen to be sufficiently small later. Given $x \in \mathfrak{S}_{\beta, q}$, we define $\nu$ to be the product measure on $\Omega$ where the colour of each vertex is sampled independently from the distribution given by $x$. We denote the Gibbs measure $\mu$ of the Curie--Weiss--Potts model and the product measure $\nu$, conditioned on $\widetilde{\Omega}(x, r)$ by, respectively,
\begin{equation}
    \tilde{\mu}(\cdot) := \mu(\, \cdot \mid \widetilde{\Omega}(x, r))
    \quad \text{and} \quad
    \tilde{\nu}(\cdot) := \nu(\, \cdot \mid \widetilde{\Omega}(x, r)).
\end{equation}

\begin{theorem} \label{WDbound}
Suppose that $\beta \geq \beta_s$, and $r$ is a sufficiently small constant. For any $x \in \mathfrak{S}_{\beta, q}$, let $\widetilde{X} \in \widetilde{\Omega}(x, r)$ and $\widetilde{Y} \in \widetilde{\Omega}(x, r)$ be random configurations distributed according to the conditional measures $\tilde{\mu}$ and $\tilde{\nu}$ respectively. Then there exists a constant $\theta^* > 0$, depending on $\beta$, $q$ and $r$, such that for any function $h: \widetilde{\Omega}(x, r) \to \reals$,
\begin{equation*}
    \lvert \mathbb{E} h(\widetilde{X}) - \mathbb{E} h(\widetilde{Y}) \rvert \leq \norm{L(h)}_\infty \theta^* \sqrt{N}.
\end{equation*}
\end{theorem}

A similar observation on the optimality of Theorem~\ref{WDbound} as in Remark~\ref{rmk:optimality_limit} can also be formulated by centring the vector of proportions around the chosen $x \in \mathfrak{S}_{\beta, q}$ and using known central limit-type results for the conditional measure $\tilde{\mu}$~\cite[Theorem~2.5]{ellis1990limit}.

As one would expect, it can be shown that in Theorem~\ref{WDbound}, $\theta^* \to 0$ as $\beta \to \infty$ (assuming that $N \to \infty$ and $r \to 0$ together at appropriate rates); see Remark~\ref{rmk:thetastar_limit}. Similarly, it can be shown that in Theorem~\ref{WDbound_high}, $\theta^* \to 0$ as $\beta \to 0$.

The proofs for Theorems~\ref{WDbound_high} and~\ref{WDbound} are given in Section~\ref{CWP}. As discussed before, the key technical difficulty is that the Glauber dynamics for the Curie--Weiss--Potts model is not contracting on the entire state space, but only \emph{locally} around each of the points in $\mathfrak{S}_{\beta, q}$.
To address this, we will analyse the \emph{restricted Glauber dynamics}---this is a Markov chain, denoted by $(\tilde{\sigma}_t)_{t \geq 0}$, which has $\tilde{\mu}$ as its stationary distribution, and evolves like the usual Glauber dynamics except that any moves out of $\widetilde{\Omega}(x, r)$ are rejected. More precisely, its transitions are as follows: given the current configuration $\tilde{\sigma}_t \in \widetilde{\Omega}(x, r)$,
\begin{enumerate}[label=(\arabic*)]
    \item Generate a new configuration $\sigma' \in \Omega$ according to the usual Glauber dynamics.
    \item If $\sigma' \in \widetilde{\Omega}(x, r)$, then set $\tilde{\sigma}_{t+1} = \sigma'$. Otherwise,  set $\tilde{\sigma}_{t+1} = \tilde{\sigma}_{t}$ (i.e.\ the move is rejected).
\end{enumerate}
We say that the restricted Glauber dynamics $\tilde{\sigma}_t$ is \emph{on the boundary} if there is a possible transition that can lead to rejection.
As a key ingredient for proving Theorems~\ref{WDbound_high} and~\ref{WDbound}, we also prove the following bound on the mixing time of the restricted Glauber dynamics in Section~\ref{CWP}.

\begin{theorem} \label{mixingtime}
Suppose that $\beta \geq \beta_s$, and $r$ is a sufficiently small constant. For any $x \in \mathfrak{S}_{\beta, q}$, let $t^x_{\mathrm{mix}}(\epsilon)$ be the mixing time of the Glauber dynamics for the Curie--Weiss--Potts model restricted to $\widetilde{\Omega}(x, r)$. Then
\[
    t^x_{\mathrm{mix}} := t^x_{\mathrm{mix}}\left( 1/4 \right) = O\left( N \log(N) \right).
\]
\end{theorem}

An analogous result for the mean-field Ising model with spins on the complete graph (where $\beta_s = \beta_c = 1$) was established in~\cite{levin2010glauber}, where a variant of the restricted Glauber dynamics was shown to have mixing time $O(N \log(N))$. 
The proof of this result exploits the symmetry of the distribution of the normalised magnetisation under the corresponding Gibbs measure, which is unique to the mean-field Ising model and does not apply in our setting.

\subsection{Related works}

\subsubsection*{Limit theorems for the Curie--Weiss--Potts model.}
For the Curie--Weiss--Potts model $X \sim \mu$, the rate of convergence of the proportions vector $S(X)$ is studied in~\cite{EichelsbacherMartschink2015}, quantifying the central limit theorems obtained in~\cite{ellis1990limit}. In particular, it is proved in~\cite[Theorem~1.3]{EichelsbacherMartschink2015} that if $W_X = \sqrt{N} \left( S(X) - x \right)$ is the scaled and centred proportions vector for any $x \in \mathfrak{S}_{\beta,q}$, and $\mathbf{\Sigma}^{1/2} Z$ is a multivariate normal vector with covariance matrix $\mathbf{\Sigma} = \ev{W_X W_X^{\top}}$, then in the high-temperature regime $\beta < \beta_c$,
\[
    |\mathbb{E} g(W_X) - \mathbb{E} g(\mathbf{\Sigma}^{1/2} Z)| \leq C N^{-1/2}
\]
for every three times differentiable function $g: \reals^q \to \reals$ with bounded derivatives. In the low-temperature regime $\beta \geq \beta_c$, a similar result with the same $O(N^{-1/2})$ rate is proved for the corresponding measures conditional on $S(X) \in \widetilde{\Omega}(x, r)$ in~\cite[Theorem~1.5]{EichelsbacherMartschink2015}.
These results can be compared to Theorems~\ref{WDbound_high} and~\ref{WDbound} respectively (also see Remark~\ref{rmk:optimality_limit}).

\subsubsection*{Glauber dynamics for the Curie--Weiss--Potts model.}

An alternative proof of the rapid mixing of the Glauber dynamics for the (generalised) Curie--Weiss--Potts model in the subcritical regime $\beta < \beta_s$ when there is a unique equilibrium macrostate $x = \hat{e}$ is given in~\cite{kovchegov2015rapid} using the aggregate path coupling method. Some of the ideas used in the analysis, extended to the low-temperature regime where there are multiple, possibly asymmetric equilibria $x \in \mathfrak{S}_{\beta, q}$, appear in the proofs of Theorems~\ref{WDbound_high}, \ref{WDbound}, and~\ref{mixingtime} to show that the Glauber dynamics restricted to $\widetilde{\Omega}(x, r)$ is contracting as long as it is sufficiently close to $x$.

The mixing time of the Glauber dynamics for the Curie--Weiss--Potts model, conditional on being close to an equilibrium macrostate, in Theorem~\ref{mixingtime} is related to different notions of mixing for when there exist metastable states that take the chain exponentially long to escape.
In~\cite[Theorem~4]{cuff2012}, it is shown that in the subcritical regime $\beta_s \leq \beta < \beta_c$, the Glauber dynamics mixes rapidly with cutoff at $(2(1 - 2 \beta / q))^{-1} N \log(N)$ after excluding a set of initial configurations with probability exponentially small in $N$---this is called \emph{essential mixing}. A set of intricate couplings is used in this paper to obtain the correct constant for the cutoff time.

The dynamics in the low-temperature regime $\beta \geq \beta_c$, which requires a more delicate treatment of asymmetric equilibrium macrostates in the presence of phase coexistence, is studied independently in our paper and the concurrent work~\cite{blanca2024meanfield}.
The key idea behind the rapid mixing results for the Glauber dynamics for the Curie--Weiss--Potts model, following~\cite{cuff2012}, is to control the drift of the proportions chain $S(\sigma_t)$ induced by the dynamics.
It is shown in~\cite{blanca2024meanfield} that with a carefully chosen product measure initialisation that is not too close to a saddle point of the free energy, the Glauber dynamics mixes rapidly in $O(N \log(N))$ time in the low-temperature regime.
Their proof technique relies on approximating projections of the high-dimensional proportions chain by tractable one-dimensional processes.
This type of mixing, starting from a specific initial distribution, is called \emph{metastable mixing} (also see~\cite{lubetzky2021fast, gheissari2022low} and \cite{bresler2024metastable} where this notion has been studied for the Ising and exponential random graph models).
Our focus, on the other hand, is on bounding the worst-case mixing time of the Glauber dynamics restricted near a particular equilibrium macrostate. The main difficulty is to control the behaviour of the dynamics at the boundary, and we adopt a matrix representation for the drift of the proportions chain that is well-suited for this analysis.

An important property of the usual worst-case mixing time~\eqref{eq:defn_mixingtime}, which Theorem~\ref{mixingtime} bounds, is that the total variation distance to stationarity decays exponentially fast beyond this time (see~\cite[Equation~(4.33)]{levin2017markov}). This plays an essential role for our purpose of distributional approximation using Stein's method (Theorem~\ref{WDbound}). This property is not exhibited by the mixing time from a specific initial distribution, i.e.\ corresponding to the notions of essential mixing or metastable mixing described above (see the discussion following Definition~1.4 in~\cite{bresler2024metastable}).

\subsubsection*{Temporal and spatial mixing.}
Deep connections have been found between rapid mixing and the spatial properties of spin systems on a lattice~\cite{holley1985possible, aizenman1987rapid, stroock1992logarithmic, martinelli1994approachI, martinelli1994approachII, dyer2004mixing}, and it is widely believed that a correspondence between temporal and spatial mixing holds for many models in statistical mechanics. By now, it is well understood that a contracting Markov chain mixes rapidly. The path coupling method~\cite{bubley1997path} is a fundamental tool for establishing that a chain is contracting by showing that~\eqref{eq:defn_contracting} holds for all pairs of neighbouring configurations. This condition has an intimate relation with Dobrushin's condition for the uniqueness of the Gibbs measure (see, e.g.,~\cite{ollivier2010survey, paulin2016mixing, dorlas2023wasserstein}), which is known to imply exponential decay of correlations~\cite{dobrushin1985constructive, gross1979decay, kunsch1982decay}; i.e.\ each spin is asymptotically independent of the vertices far away.

\subsection{Organisation}

The rest of the paper is organised as follows. Section~\ref{sec:prelim} contains preliminaries on the Glauber dynamics and Stein's method.
Section~\ref{sec:potts_bounded_degree} describes the proof of Theorem~\ref{thm:potts_approx_result} for the Potts model on general bounded-degree graphs.
In Section~\ref{CWP}, a detailed analysis of the Curie--Weiss--Potts model and its restricted Glauber dynamics is provided, and the proofs of Theorems~\ref{WDbound_high}, \ref{WDbound}, and~\ref{mixingtime} are given.
Finally, the proofs of some deferred technical lemmas appear in Appendix~\ref{app:deferred_proofs}.

\section{Preliminaries} \label{sec:prelim}

\subsection{Glauber dynamics for the Potts model}

Consider the Potts model with inverse temperature $\beta > 0$ on a graph $G = (V, E)$ on $N$ vertices with $q \geq 3$ colours. For any vertex $v \in V$, let $\mathcal{N}_v := \{ u \in V : (u, v) \in E \}$ denote the neighbours of $v$. For any configuration $\sigma \in [q]^V$, let
\begin{equation} \label{general_magnetization_vector}
    S_v(\sigma) := (S_v^{(1)}(\sigma), \ldots, S_v^{(q)}(\sigma)) \in \reals^q
\end{equation}
be the vector whose coordinates
\begin{equation} \label{general_magnetization_vector_cor}
    S_v^{(k)}(\sigma) := \frac{1}{N} \sum_{u \in \mathcal{N}_v} \indicator{ \sigma(u)=k }
\end{equation}
indicate the (scaled) proportion of neighbours of $v$ with each of the $q$ colours.\footnote{This definition essentially coincides with the vector of proportions defined in~\eqref{eq:CWP_vector_of_proportions} for the Curie--Weiss--Potts model, where it is convenient to include the colour of the vertex $v$ in the count due to symmetry.} Moreover, for $s \in \reals^q$, define the vector
\begin{equation} \label{updatingprobvector}
    g_{\beta}(s) := \left( g^{(1)}_{\beta}(s), \ldots, g^{(q)}_{\beta}(s) \right) \in \reals^q_+,
\end{equation}
where for each colour $k \in [q]$,
\begin{equation} \label{updatingprob}
    g_{\beta}^{{(k)}}(s) := \frac{e^{2\beta s^{{(k)}}}}{\sum_{j=1}^{q} e^{2\beta s^{{(j)}}}}.
\end{equation}
In other words, we can identify $g_\beta$ with the softmax function which maps vectors in $\reals^q$ to the $q$-dimensional probability simplex $\mathcal{S} := \{ x \in \mathbb{R}_+^q : \norm{x}_1 = 1 \}$.
Furthermore, we denote $\sigma^{(v, k)}$ to be the configuration obtained from $\sigma$ by recolouring the vertex $v$ with colour $k$: that is,
\begin{equation}
    \sigma^{(v, k)}(u) =
    \begin{cases}
        k         & \text{if } u = v, \\
        \sigma(u) & \text{if } u \ne v.
    \end{cases}
\end{equation}
Using this notation, we can specify the Glauber dynamics for the Gibbs measure $\mu$ of the Potts model more precisely: given the current configuration $\sigma$, the next step selects a vertex $v$ uniformly at random and recolours it with a new colour $k \in [q]$, selected with probability
\begin{equation} \label{eq:potts_condspindist}
    \mu_v(k \mid \sigma_t) = g_\beta^{{(k)}}(S_v(\sigma_t)),
\end{equation}
to obtain the next configuration $\sigma^{(v, k)}$. That is, the conditional spin distribution for the Potts model is given by $\mu_v(\cdot \mid \sigma) = g_\beta(S_v(\sigma))$ at each vertex $v$. If $\nu$ is another measure on $[q]^V$, we will also denote the total variation distance between its conditional spin distribution $\nu_v(\cdot \mid \sigma)$ and $\mu_v(\cdot \mid \sigma)$ by
\begin{equation} \label{eq:condspindist_tv}
    T_v(\sigma) := \norm{\mu_v(\cdot \mid \sigma) - \nu_v(\cdot \mid \sigma)}_{\mathrm{TV}},
\end{equation}
and collect these values in the vector $T(\sigma) := (T_v(\sigma))_{v \in V} \in \reals^V_+$.

\subsection{Stein's method for approximating the stationary distributions of Glauber dynamics} \label{sec:prelim_steins}

In this section, we provide an overview of the proof of Theorem \ref{thm:approx_thm_general_contracting}, which uses the generator approach of Stein's method, in order to establish some intermediary results that we will use when the chain is not globally contracting.

The idea behind the generator approach to compare two random vectors $X \sim \mu$ and $Y \sim \nu$ on a finite metric space $(\Sigma, d)$ is to find generators $\mathcal{A}_\mu$ and $\mathcal{A}_\nu$ for Markov processes with stationary distributions $\mu$ and $\nu$ respectively, and for any test function $h: \Sigma \to \reals$, solve the Stein equation
\begin{equation} \label{eq:stein_eqn}
    \mathcal{A}_\mu f_h(\sigma) = h(\sigma) - \E h(X)
\end{equation}
for $f_h: \Sigma \to \reals$. Since $\E \mathcal{A}_\nu f(Y) = 0$ if and only if $Y \sim \nu$, the two measures $\mu$ and $\nu$ can be compared in terms of the solution to the Stein equation by
\begin{equation} \label{eq:stein_eqn_comparison}
    |\mathbb{E} h(X) - \mathbb{E} h(Y)| = |\mathbb{E} \mathcal{A}_\mu f_h(Y) - \mathbb{E} \mathcal{A}_\nu f_h(Y)| .
\end{equation}
In particular, if $P$ and $Q$ are Markov chains on $\Sigma$ with stationary distributions $\mu$ and $\nu$ (or more precisely, their transition kernels), then the Markov processes can be chosen to be the continuous-time Markov chains with Exponential rate one holding times and jump probabilities induced by $P$ and $Q$, which have generators $\mathcal{A}_\mu = P - I$ and $\mathcal{A}_\nu = Q - I$ respectively. If $(\bar{X}_t)_{t \geq 0}$ denotes the continuous-time Markov chain with stationary distribution $\mu$, then it is known that the Stein equation has the following well-defined solution if $P$ is irreducible (e.g.\ see~\cite{reinert2005three}):
\begin{equation} \label{eq:stein_eqn_soln}
    f_h(\sigma) \coloneqq -\int^\infty_0 \cev{h(\bar{X}_t) - \mathbb{E} h(X)}{\bar{X}_0 = \sigma} \dd t , \quad \sigma \in \Sigma .
\end{equation}
If $P$ is contracting according to~\eqref{eq:defn_contracting}, then it was shown by~\cite{blanca2022mixing} that the irreducibility assumption can be removed (i.e.\ $f_h$ is convergent and remains well-defined).

Note that $(P - Q)f_h(\sigma) = \ev{f_h(X^\sigma_1) - f_h(Y^\sigma_1)}$, where $(X^\sigma_1, Y^\sigma_1)$ is a coupling of $P(\sigma, \cdot)$ and $Q(\sigma, \cdot)$. Thus, by definition of the Wasserstein distance~\eqref{eq:wasserstein_defn}, $|(P - Q)f_h(\sigma)| \leq L_d(f_h) \cdot d_W(P(\sigma, \cdot), Q(\sigma, \cdot))$. Hence, by substituting $\mathcal{A}_\mu = P - I$ and $\mathcal{A}_\nu = Q - I$ into~\eqref{eq:stein_eqn_comparison}, we have
\begin{equation} \label{eq:approx_wasserstein_Lfh}
    |\mathbb{E} h(X) - \mathbb{E} h(Y)| 
    = \mathbb{E} |(P - Q) f_h(Y)|
    \leq L_d(f_h) \cdot \ev{d_W(P(Y, \cdot), Q(Y, \cdot))}. 
\end{equation}
The following argument furnishes the key estimate required to bound $L_d(f_h)$, the Lipschitz constant of $f_h$. Denote the independent jump times of the continuous-time chain $(\bar{X}_t)_{t \geq 0}$ by $0 = T_0 < T_1 < T_2 < \dots$, and the embedded jump chain by $(X_\ell)_{\ell \in \naturals}$ (i.e.\ $\bar{X}_t = X_{\ell}$ for $t \in [T_\ell, T_{\ell+1})$). By reducing to the discrete skeleton of the continuous-time chain, $f_h$ satisfies
\begin{equation}
\begin{aligned}
    |f_h(\sigma) - f_h(\tau)|
    &\leq \int^\infty_0 \left| \cev{h(\bar{X}_t)}{\bar{X}_0 = \sigma} - \cev{h(\bar{X}_t)}{\bar{X}_0 = \tau} \right| \dd t \\
    &\leq \sum_{\ell=0}^\infty \mathbb{E} \int^{T_{\ell+1}}_{T_\ell} \left| h(X^\sigma_\ell) - h(X^\tau_\ell) \right| \dd t.
\end{aligned}
\end{equation}
Since the jump times are independent of the discrete skeleton and have mean one,
\begin{equation} \label{eq:fh_lipschitz_bound}
    |f_h(\sigma) - f_h(\tau)|
    = \sum_{\ell=0}^\infty \mathbb{E} \left| h(X^\sigma_\ell) - h(X^\tau_\ell) \right| \\
    \leq L_d(h) \sum_{\ell=0}^\infty d_W(P^\ell(\sigma, \cdot), P^\ell(\tau, \cdot)) ,
\end{equation}
where the inequality follows from the definition of the Wasserstein distance~\eqref{eq:wasserstein_defn}. Hence, Theorem~\ref{thm:approx_thm_general_contracting} follows from combining~\eqref{eq:approx_wasserstein_Lfh} and the following estimate for $L_d(f_h)$, obtained from a straightforward application of the contracting assumption~\eqref{eq:defn_contracting}. Indeed, since $P$ is contracting,
\[
    d_W(P^\ell(\sigma, \cdot), P^\ell(\tau, \cdot)) \leq \kappa^\ell \cdot d(\sigma, \tau)
\]
for all $\ell \geq 0$. Therefore, from~\eqref{eq:fh_lipschitz_bound}, we have
\[
    |f_h(\sigma) - f_h(\tau)|
    \leq L_d(h) \sum_{\ell=0}^\infty \kappa^\ell \cdot d(\sigma, \tau)
    = \frac{L_d(h)}{1 - \kappa} d(\sigma, \tau),
\]
which shows that $L_d(f_h) \leq L_d(h) / (1 - \kappa)$.

To conclude this section, we extract some more specific estimates from the general argument above when $\Sigma = [q]^V$ is the set of configurations on an underlying graph $G = (V, E)$, $d = d_H$ is the Hamming distance, and $P$ and $Q$ are the Glauber dynamics for $\mu$ and $\nu$. In this setting, the generator $\mathcal{A}_\mu$ for the continuous-time Glauber dynamics induced by $P$ takes the form
\begin{equation} \label{eq:glauber_generator}
    \mathcal{A}_\mu f(\sigma) = \frac{1}{N} \sum_{v \in V} \left[ \sum_{k \in [q]} \mu_v( k \mid \sigma ) ( f(\sigma^{(v, k)}) - f(\sigma) ) \right] , \quad \sigma \in [q]^V,
\end{equation}
recalling that $\sigma^{(v, k)}$ denotes the configuration obtained from $\sigma$ by recolouring the vertex $v$ with colour $k$, and $\mu_v(\cdot \mid \sigma)$ is the conditional spin distribution. The generator $\mathcal{A}_\nu$ for the continuous-time chain induced by $Q$ takes the same form as~\eqref{eq:glauber_generator} with $\nu_v(\cdot \mid \sigma)$ in place of $\mu_v(\cdot \mid \sigma)$.

The following lemmas will be used in the proof of the approximation results for the Curie--Weiss--Potts model in Section~\ref{CWP}.
By putting in the specific form of the generators $\mathcal{A}_\mu$ and $\mathcal{A}_\nu$ from above into the Stein's method bound~\eqref{eq:stein_eqn_comparison}, we immediately deduce the following analogue of~\eqref{eq:approx_wasserstein_Lfh}, specialised to the Glauber dynamics:

\begin{lemma} \label{lem:continuous_glauber_generator}
Let $P$ and $Q$ be the Glauber dynamics on $[q]^V$ with stationary distributions $\mu$ and $\nu$ respectively. For any $h: [q]^V \to \reals$, we have
\begin{equation} \label{eq:glauber_generator_1}
\begin{aligned}
    |\mathbb{E} h(X) - \mathbb{E} h(Y)|
    &\leq \frac{1}{N} \sum_{v \in V} \sum_{k \in [q]} \ev{ \left| \mu_v( k \mid Y ) - \nu_v( k \mid Y ) \right| \cdot \left| f_h(Y^{(v,k)}) - f_h(Y) \right| } .
\end{aligned}
\end{equation}
\end{lemma}

Furthermore, the following lemma will be used to bound the differences of the solution to the Stein equation $f_h$, which appears in Lemma~\ref{lem:continuous_glauber_generator}, when the chain $P$ is not contracting:

\begin{lemma} \label{lem:glauber_bound_fh_diff}
Let $P$ be the Glauber dynamics on $[q]^V$ with stationary distribution $\mu$. For any $\sigma, \tau \in [q]^V$, let $(X^\sigma_\ell, X^\tau_\ell)_{\ell \in \naturals}$ be any sequence of couplings of the $\ell$-step distributions of $P$ starting from $\sigma$ and $\tau$ respectively. Then for any $h: [q]^V \to \reals$, we have
\[
    |f_h(\sigma) - f_h(\tau)| \leq \norm{L(h)}_\infty \sum_{\substack{\ell \geq 0 \\ v \in V}} \Prob{ X_\ell^\sigma(v) \ne X_\ell^\tau(v) } . 
\]
\end{lemma}

\begin{proof}
By definition of the Wasserstein distance~\eqref{eq:wasserstein_defn}, $d_W(P^\ell(\sigma, \cdot), P^\ell(\tau, \cdot))$ can be upper bounded by $\mathbb{E} d_H(X^\sigma_\ell, X^\tau_\ell) = \sum_{v \in V} \Prob{ X_\ell^\sigma(v) \ne X_\ell^\tau(v) }$, where $(X^\sigma_\ell, X^\tau_\ell)$ is any coupling of the $\ell$-step distributions of $P$ starting from $\sigma$ and $\tau$. 
Therefore, since $h: [q]^V \to \reals$ is $\norm{L(h)}_\infty$-Lipschitz with respect to $d_H$, we obtain the claimed bound by continuing from~\eqref{eq:fh_lipschitz_bound}.
\end{proof}

\section{Potts model on general bounded-degree graphs} \label{sec:potts_bounded_degree}

In this section, we prove Theorem~\ref{thm:potts_approx_result} for the Potts model on a graph $G = (V, E)$ with $N$ vertices and maximum degree $\Delta$. First, the following lemma establishes a condition for when the Glauber dynamics for the Potts model is contracting. The proof of this result is essentially embedded in~\cite{ullrich2014mixing}, but we will describe it for completeness.

\begin{lemma} \label{lem:potts_bounded_degree_contracting}
Let $P$ be the Glauber dynamics for the Potts model with inverse temperature $\beta$ on a graph $G = (V, E)$ with $N$ vertices and maximum degree $\Delta$. If $\Delta \tanh(\beta / N) < 1$, then for all $\sigma, \tau \in [q]^V$, there exists a coupling $(X^\sigma_1, X^\tau_1)$ of the one-step distributions of $P$ starting from $\sigma, \tau$ such that 
\begin{equation} \label{eq:lem_potts_bounded_degree_contracting_1}
    \ev{d_H(X^\sigma_1, X^\tau_1)} \leq \left( 1 - \frac{1 - \Delta \tanh(\beta / N)}{N} \right) d_H(\sigma, \tau).
\end{equation}
That is, $P$ is contracting with respect to the Hamming distance $d_H$ according to~\eqref{eq:defn_contracting} with rate $\kappa = 1 - N^{-1} (1 - \Delta \tanh(\beta / N))$.
\end{lemma}

\begin{proof}
Recall from~\eqref{eq:cond_spin_dist} that $\mu_v(\cdot \mid \sigma)$ is the conditional spin distribution of $\mu$ at $v$, given $\sigma$.
It is shown in the proof of~\cite[Theorem~2.13]{ullrich2014mixing}\footnote{Note that the inverse temperature $\beta$ in~\cite{ullrich2014mixing} corresponds to $2 \beta / N$ under our scaling for the Gibbs measure.} that for all $u, v \in V$, the so-called influence of $u$ on $v$ for the chain $P$ satisfies
\begin{equation} \label{eq:lem_potts_bounded_degree_contracting_2}
    \widehat{R}_{u, v} := \max_{\sigma(w) = \tau(w) \,\forall w \ne u} \norm{\mu_v(\cdot \mid \sigma) - \mu_v(\cdot \mid \tau)}_{\mathrm{TV}} \leq \tanh\left( \frac{\beta}{N} \right) \indicator{(u, v) \in E}.
\end{equation}
Here, the maximum is taken over all pairs $\sigma, \tau \in [q]^V$ that only possibly differ at $u$.
The rest of the proof follows from a standard path coupling calculation~\cite{dyer2009spin, BordewichGreenhillPatel14} for which it suffices to show that~\eqref{eq:lem_potts_bounded_degree_contracting_1} holds for all $\sigma, \tau \in [q]^V$ with $d_H(\sigma, \tau) = 1$.
Given a pair $\sigma, \tau \in [q]^V$ that only differs at a single vertex, say $u$, let $(X^\sigma_1, X^\tau_1)$ be a coupling of the one-step distributions of $P$ starting from $\sigma$ and $\tau$ which samples the same vertex $v$ uniformly at random and minimises the probability that the selected colour differs.
Observe that the Hamming distance between the two chains decreases by one if $v = u$ with probability $1 / N$, and increases by one if the sampled colour differs, which occurs with probability at most $\widehat{R}_{u, v} \leq \tanh(\beta / N)$ by~\eqref{eq:lem_potts_bounded_degree_contracting_2}. Thus, $\ev{d_H(X^\sigma_1, X^\tau_1)} \leq 1 - N^{-1} (1 - \Delta \tanh(\beta / N))$.
\end{proof}

\begin{remark}
By using a standard coupling argument, Lemma~\ref{lem:potts_bounded_degree_contracting} implies that if $\Delta \tanh(\beta / N) < 1$ (or more simply $\beta < N / \Delta$), then the Glauber dynamics for the Potts model mixes rapidly with
\[
    t_{\mathrm{mix}}(\epsilon) \leq \frac{N \log (N \epsilon^{-1})}{1 - \Delta \tanh(\beta / N)}.
\]
This is the best known range of $\beta$ for rapid mixing of the Glauber dynamics when $q$ is fixed and $N$ is large. In the regime where $q$ is large, it has been shown in~\cite{BordewichGreenhillPatel14, blanca2022mixing} that the Glauber dynamics also has optimal $O(N \log(N))$ mixing time if $\beta < \frac{1}{\Delta} \log\left( \frac{q-1}{\Delta} \right)$ or $\beta < (1 - o_q(1)) \frac{\log(q)}{\Delta - 1}$ (where $o_q(1) \to 0$ as $q \to \infty$).
\end{remark}

Next, the following result specialises the general approximation result in Theorem~\ref{thm:approx_thm_general_contracting} to the Glauber dynamics on $[q]^V$.
Recall that $T(\sigma) = (T_v(\sigma))_{v \in V}$, defined in~\eqref{eq:condspindist_tv}, denotes the total variation distances between the conditional spin distributions of $\mu$ and $\nu$, given $\sigma$.

\begin{theorem} \label{thm:approx_thm_glauber_contracting}
Let $P$ and $Q$ be the Glauber dynamics on $[q]^V$ with stationary distributions $\mu$ and $\nu$ respectively, and let $X \sim \mu$ and $Y \sim \nu$ be random configurations. If $P$ is contracting with respect to the Hamming distance $d_H$ according to~\eqref{eq:defn_contracting} for some $0 \leq \kappa < 1$, then for any function $h: [q]^V \to \reals$,
\begin{equation} \label{eq:approx_thm_glauber_contracting_1}
    |\mathbb{E} h(X) - \mathbb{E} h(Y)| \leq \frac{\norm{L(h)}_\infty}{N(1 - \kappa)} \mathbb{E} \norm{T(Y)}_1 .
\end{equation}
\end{theorem}

\begin{proof}
Recall that $h$ is $\norm{L(h)}_\infty$-Lipschitz with respect to $d_H$. By definition of the Wasserstein distance~\eqref{eq:wasserstein_defn}, $d_W(P(\sigma, \cdot), Q(\sigma, \cdot))$ can be upper bounded by $\mathbb{E}\left[ d_H(X^\sigma_1, Y^\sigma_1) \right]$, where $(X^\sigma_1, Y^\sigma_1)$ is the following coupling of the one-step distributions of $P$ and $Q$ starting from any configuration $\sigma \in [q]^V$:
\begin{enumerate}[label=(\arabic*)]
    \item Pick the same vertex $v \in V$ uniformly at random.
    \item Update the spin at $v$ to obtain new configurations $X^\sigma_1$ and $Y^\sigma_1$ according to an optimal coupling of the conditional spin distributions $\mu_v( \cdot \mid \sigma )$ and $\nu_v( \cdot \mid \sigma )$ respectively (i.e.\ which minimises the probability that the selected colour differs).
\end{enumerate}
If $v$ is selected, the probability that $X^\sigma_1(v) \ne Y^\sigma_1(v)$, or equivalently $d_H(X^\sigma_1, Y^\sigma_1) = 1$, is equal to $T_v(\sigma) = \norm{\mu_v( \cdot \mid \sigma) - \nu_v( \cdot \mid \sigma)}_{\mathrm{TV}}$. Thus, from Theorem~\ref{thm:approx_thm_general_contracting}, we have
\[
    |\mathbb{E} h(X) - \mathbb{E} h(Y)|
    \leq \frac{\norm{L(h)}_\infty}{1 - \kappa} \ev{\frac{1}{N} \sum_{v \in V} T_v(Y)} 
    = \frac{\norm{L(h)}_\infty}{N(1 - \kappa)} \mathbb{E} \norm{T(Y)}_1,
\]
as desired.
\end{proof}

Now suppose that $\mu$ is the Potts model and that $\nu$ is a product measure where each spin is independently distributed according to the probabilities $p_v := (p_v^{(k)})_{k \in [q]}$ at each vertex $v \in V$.
Note that for any configuration $\sigma \in [q]^V$, $\nu_v(\cdot \mid \sigma) = p_v$, and from~\eqref{eq:potts_condspindist}, $\mu_v(\cdot \mid \sigma) = g_\beta(S_v(\sigma))$, where $g_\beta$ is the softmax function from~\eqref{updatingprob} and $S_v(\sigma)$ indicates the (scaled) proportions of neighbours of $v$ with each colour from~\eqref{general_magnetization_vector}.
From Theorem~\ref{thm:approx_thm_glauber_contracting}, $Y \sim \nu$ offers a good approximation of the Potts model if $\E \norm{T(Y)}_1$ is small. Hence, we seek probabilities that satisfy $p_v^{(k)} \approx \mu_v(k \mid Y)$ to match the transition probabilities of the corresponding Glauber dynamics---a natural choice is to solve the system of equations
\begin{equation} \label{eq:potts_meanfield}
    p_v^{(k)} = g_\beta^{(k)}(\mathbb{E}[S_v(Y)]) \quad \text{for all} \quad k \in [q],\, v \in V.
\end{equation}
The condition~\eqref{eq:potts_meanfield} is a mean-field approximation and generalises a similar condition for the Ising model in~\cite[Equation~(1.2)]{reinert2019approximating}. An important observation that we will use is that setting $p^{(k)}_v = 1 / q$ for all $k \in [q]$, $v \in V$ (i.e.\ $\nu$ consists of i.i.d.\ uniform spins) always produces a solution to~\eqref{eq:potts_meanfield}.

By applying Theorem~\ref{thm:approx_thm_glauber_contracting}, we establish the following approximation result for comparing the Potts model with a product measure $\nu$ that satisfies the mean-field condition~\eqref{eq:potts_meanfield}.

\begin{lemma} \label{lem:potts_approx_meanfield}
Let $X \sim \mu$ be distributed according to the Potts model with inverse temperature $\beta$ on a graph $G = (V, E)$ with $N$ vertices, and $Y \sim \nu$ be distributed according to a product measure $(p_v)_{v \in V}$ on $[q]^V$ that satisfies~\eqref{eq:potts_meanfield}. If the Glauber dynamics for $X$ is contracting with respect to the Hamming distance according to~\eqref{eq:defn_contracting} for some $0 \leq \kappa < 1$, then for any $h: [q]^V \to \reals$,
\[
    | \mathbb{E} h(X) - \mathbb{E} h(Y) |
    \leq \norm{L(h)}_\infty \frac{\beta \sqrt{q}}{N(1 - \kappa)} \sum_{v \in V} \E \norm{S_v(Y) - \ev{S_v(Y)}}_2.
\]
\end{lemma}

\begin{proof}
The bound follows from Theorem~\ref{thm:approx_thm_glauber_contracting} once we show that
\begin{equation} \label{eq:potts_approx_meanfield_pf1}
    \E \norm{T(Y)}_1 \leq \beta \sqrt{q} \sum_{v \in V} \E \norm{S_v(Y) - \E S_v(Y)}_2,
\end{equation}
where $T(Y) = (T_v(Y))_{v \in V}$ with $T_v(Y) = \norm{\mu_v(\cdot \mid Y) - \nu_v(\cdot \mid Y)}_{\mathrm{TV}}$. Fix $v \in V$. 
Since the conditional spin distribution of $P$ is given by $\mu_v(\cdot \mid \sigma) = g_\beta(S_v(\sigma))$ from~\eqref{eq:potts_condspindist}, and $\nu_v(k \mid \sigma) = p_v^{(k)}$ for the random configuration $Y$ with independent spins that satisfy~\eqref{eq:potts_meanfield}, we have
\[
    \mu_v( k \mid Y ) - \nu_v( k \mid Y ) = g_\beta(S_v(Y)) - g_\beta(\ev{S_v(Y)}) , \quad k \in [q] .
\]
By using the equivalence of the $\ell_1$ and $\ell_2$ norms in $\mathbb{R}^q$, and the fact that the softmax function $g_\beta$ is $2 \beta$-Lipschitz with respect to the $\ell_2$ norm~\cite[Proposition 4]{GaoPavel2018}, we deduce that
\[
\begin{aligned}
    \norm{ g_\beta(S_v(Y)) - g_\beta(\ev{S_v(Y)}) }_1 
    &\leq \sqrt{q} \norm{ g_\beta(S_v(Y)) - g_\beta(\ev{S_v(Y)}) }_2 \\
    &\leq 2 \beta \sqrt{q} \norm{ S_v(Y) - \ev{S_v(Y)} }_2 .
\end{aligned}
\]
Combining the previous displayed equations implies that
\[
\begin{aligned}
    T_v(Y) = \frac{1}{2} \norm{ g_\beta(S_v(Y)) - g_\beta(\ev{S_v(Y)}) }_1
    \leq \beta \sqrt{q} \norm{ S_v(Y) - \ev{S_v(Y)} }_2.
\end{aligned}
\]
Hence, by summing over all the vertices, we deduce that~\eqref{eq:potts_approx_meanfield_pf1} holds, as desired.
\end{proof}

By combining Lemmas~\ref{lem:potts_bounded_degree_contracting} and~\ref{lem:potts_approx_meanfield}, we can now prove Theorem~\ref{thm:potts_approx_result}.

\begin{proof}[Proof of Theorem~\ref{thm:potts_approx_result}]
By Lemma~\ref{lem:potts_bounded_degree_contracting}, the Glauber dynamics for the Potts model is contracting with $\kappa = 1 - N^{-1} (1 - \Delta \tanh(\beta / N))$ whenever $\Delta \tanh(\beta / N) < 1$. Hence, we may apply Lemma~\ref{lem:potts_approx_meanfield}, where each spin of the random configuration $Y$ is independently sampled from the uniform distribution $p = (p^{(k)})_{k \in [q]}$ with $p^{(k)} = 1/q$ for all $k$, which satisfies~\eqref{eq:potts_meanfield}.
Thus, it suffices to bound $\sum_{v \in V} \mathbb{E} \norm{ S_v(Y) - \ev{S_v(Y)} }_2$. Fix $v \in V$, and denote its degree by $\mathrm{deg}(v)$. Recall from~\eqref{general_magnetization_vector} that
\[
    S_v(Y) = \left( \frac{1}{N} \sum_{u \in \mathcal{N}_v} \indicator{ Y(u) = k } \right)_{k \in [q]} \in \reals^q .
\]
Since the colours of the neighbours of $v$ are independently distributed according to $p$, we have $N S_v(Y) \sim \mathrm{Multinomial}(\mathrm{deg}(v), p)$. Thus, $\mathrm{Var}(N S^{(k)}_v(Y)) = \mathrm{deg}(v) p^{(k)} (1 - p^{(k)})$, and
\[
    \E \left( S^{(k)}_v(Y) - \ev{S^{(k)}_v(Y)} \right)^2
    = \frac{1}{N^2} \mathrm{Var}(N S^{(k)}_v(Y))
    = \frac{\mathrm{deg}(v)}{N^2} p^{(k)} (1 - p^{(k)})
\]
for all $k \in [q]$. Consequently, by Jensen's inequality,
\begin{equation} \label{eq:potts_approx_result_pf1}
    \mathbb{E} \norm{S_v(Y) - \ev{S_v(Y)}}_2
    \leq \sqrt{\sum_{k \in [q]} \E \left( S_v^{(k)}(Y) - \ev{S_v^{(k)}(Y)} \right)^2}
    = \frac{\sqrt{\mathrm{deg}(v)}}{N} \sqrt{\sum_{k \in [q]} p^{(k)} (1 - p^{(k)})} .
\end{equation}

Since $p^{(k)} = 1/q$ for all $k$, $\sum_{k \in [q]} p^{(k)} (1 - p^{(k)}) = (q - 1) / q$. Therefore, applying Lemma~\ref{lem:potts_approx_meanfield} shows that
\begin{equation} \label{eq:potts_approx_result_pf2}
    | \mathbb{E} h(X) - \mathbb{E} h(Y) |
    \leq \norm{L(h)}_\infty \frac{\beta \sqrt{q-1}}{1 - \Delta \tanh(\beta / N)} \frac{1}{N} \sum_{v \in V} \sqrt{\mathrm{deg}(v)}.
\end{equation}
In particular, by using Jensen's inequality and the handshaking lemma, we have
\begin{equation} \label{eq:potts_approx_result_pf3}
    \frac{1}{N} \sum_{v \in V} \sqrt{\mathrm{deg}(v)} \leq \sqrt{\frac{1}{N} \sum_{v \in V} \mathrm{deg}(v)} = \sqrt{\frac{2|E|}{N}}.
\end{equation}
Combining~\eqref{eq:potts_approx_result_pf2} and~\eqref{eq:potts_approx_result_pf3} completes the proof of the theorem.
\end{proof}

\section{The Curie--Weiss--Potts model} \label{CWP}

In this section, we focus on the Curie--Weiss--Potts model on the complete graph with $N$ vertices. The analysis of the vector of proportions $S(\sigma)$, which lives in the probability simplex $\mathcal{S} = \{ x \in \reals^q_+ : \norm{x}_1 = 1 \}$, will play a crucial part for understanding the dynamics of the system. By convention, we will consider vectors $s \in \reals^q$ as row vectors.
In a continuation of the discussion in Section~\ref{sec:intro_cwp}, we begin by summarising known results about the equilibrium macrostates.

\begin{theorem}[{\cite[Theorem 2.1]{ellis1990limit}}\footnote{Note that the inverse temperature $\beta$ in~\cite{ellis1990limit} corresponds to $2 \beta$ under our scaling for the Gibbs measure.}] \label{pointofconcentration}
For $\beta > 0$, let $s_{\beta,q} \in \mathbb{R}$ be the largest solution of the equation 
\begin{equation} \label{pointofconcentration_eq1}
    s = \frac{1 - e^{-2\beta s}}{1 + (q-1) e^{-2\beta s}}.
\end{equation}
Then, the quantity $s_{\beta,q}$ is well-defined.
On the interval $[\beta_c,\infty)$, it is positive, strictly increasing and differentiable, where
\begin{equation*}
    \beta_{c} \equiv \beta_{c}(q) := \frac{q - 1}{q - 2}\log(q - 1)
\end{equation*}
satisfies $1 < \beta_c(q) < q/2$ for all $q \geq 3$,  and $\beta_c(2) = 1$. Moreover, one has $s_{\beta_c,q} = (q - 2) / (q - 1)$ and $\lim_{\beta\uparrow\infty}s_{\beta,q} = 1$ for any $q \geq 2$.

Define 
\begin{equation*}
    \check{s}_{\beta,q} := \left( \frac{1 + (q-1) s_{\beta,q}}{q}, \frac{1 - s_{\beta,q}}{q}, \ldots, \frac{1 - s_{\beta,q}}{q} \right) \in \mathcal{S}.
\end{equation*}
Let $\mathbf{T}^k : \reals^q \to \reals^q$ denote the operator which interchanges the $1$st and $k$th coordinates of a vector. Then, for $\beta = \beta_c$, we have 
\begin{equation*}
    \check{s}_{\beta_c,q} = \left( 1 - \frac{1}{q}, \frac{1}{q(q-1)}, \ldots, \frac{1}{q(q-1)} \right).
\end{equation*}
Furthermore, the set of minimisers of the function $G_{\beta}$ defined in~\eqref{eq:G_beta} is given by
\begin{align*}
    \mathfrak{S}_{\beta,q} :=
    \begin{cases}
        \{ \hat{e} \}, \quad &\text{if}\quad \beta < \beta_{c},\\
        \{ \hat{e}, \mathbf{T}^1\check{s}_{\beta_c,q}, \mathbf{T}^2\check{s}_{\beta_c,q}, \ldots, \mathbf{T}^q\check{s}_{\beta_c,q} \}, \quad &\text{if}\quad \beta = \beta_{c},\\
        \{ \mathbf{T}^1\check{s}_{\beta,q}, \ldots, \mathbf{T}^q\check{s}_{\beta,q} \}, \quad &\text{if}\quad \beta > \beta_{c},\\
    \end{cases}
\end{align*}
where $\hat{e} := (1/q, \ldots, 1/q) \in \mathcal{S}$ is the equiproportionality vector. For $\beta \geq \beta_c$, the points in $\mathfrak{S}_{\beta,q}$ are all distinct.
\end{theorem}

The expression for the critical value $\beta_c$ for the Curie--Weiss--Potts model was first obtained in~\cite{wu1982potts}. For simplicity of notation, we define
\begin{equation}\label{s^*}
    s^*_{\beta,q} := \check{s}^{(1)}_{\beta,q} = \frac{1+(q-1)s_{\beta,q}}{q},
\end{equation}
where $s_{\beta,q}$ is the largest solution of the equation~\eqref{pointofconcentration_eq1} in Theorem~\ref{pointofconcentration}, and so
\begin{equation} \label{T1}
    \mathbf{T}^1\check{s}_{\beta,q} = \check{s}_{\beta,q} = \left( s^*_{\beta,q}, \frac{1-s^*_{\beta,q}}{q-1}, \ldots, \frac{1 - s^*_{\beta,q}}{q-1} \right).
\end{equation}

Next, we describe some key properties of the softmax function $g_\beta$, defined in~\eqref{updatingprobvector}, which will be important for our analysis. Observe that the gradient of the function $G_{\beta}(s)$ defined in~\eqref{eq:G_beta} is equal to $2 \beta (s - g_\beta(s))$. Therefore, the points in $\mathfrak{S}_{\beta,q}$ solve the following fixed point equation: 

\begin{lemma}[Mean-field equation] \label{g(x)=x}
For all $x \in \mathfrak{S}_{\beta,q}$, we have $g_{\beta}(x) = x$.
\end{lemma}

Lemma~\ref{g(x)=x} essentially corresponds to the mean-field condition~\eqref{eq:potts_meanfield} that previously appeared when approximating the Potts model on a general graph with a sequence of independent spins. This explains why $x \in \mathfrak{S}_{\beta, q}$ is a natural choice for the distribution of each spin of the underlying i.i.d.\ model $Y$ in Theorems~\ref{WDbound_high} and~\ref{WDbound}.

A crucial step for showing that the Glauber dynamics for the Curie--Weiss--Potts model is contracting around any equilibrium macrostate is to bound the Lipschitz constant of the function $g_{\beta}$ around any point $x \in \mathfrak{S}_{\beta,q}$. The following technical lemma describes a special property of the Jacobian matrix of $g_\beta$ at $x$, which is essential for establishing such a bound.

\begin{lemma}[Jacobian] \label{jacobian}
Let $\mathbf{J}(x)$ be the $q \times q$ Jacobian matrix of $g_\beta$ at $x$, such that the $k$th row of the matrix is given by the row vector $\nabla g_\beta^{(k)}(x)$.
Define the constants $a, a', b > 0$ by
\begin{equation*}
    a := 2\beta q s^*_{\beta,q} \frac{1 - s^*_{\beta,q}}{q-1}, \quad
    a' := 2\beta \frac{1 - s^*_{\beta,q}}{q-1}, \quad
    b := 2\beta \frac{1 - s^*_{\beta,q}}{q - 1} \left( s^*_{\beta,q} - \frac{1 - s^*_{\beta,q}}{q - 1}\right).
\end{equation*}
Then for any $x \in \mathfrak{S}_{\beta, q}$ and any $s_1, s_2 \in \mathcal{S}$, we have that
\[
    \mathbf{J}(x)(s_1 - s_2)^{\top} = \mathbf{A}(x)(s_1 - s_2)^{\top},
\]
where $\mathbf{A}(x)$ is the matrix defined as follows.
When $x = \hat{e}$, we simply have $\mathbf{A}(\hat{e}) := (2 \beta / q) \mathbf{I}$, where $\mathbf{I}$ is the $q \times q$ identity matrix---that is, 
\begin{equation*}
    \nabla g^{(k)}_{\beta}(\hat{e})(s_1 - s_2)^{\top} = \frac{2\beta}{q}(s_1^{(k)} - s_2^{(k)}) \quad \text{for all} \quad k \in [q].
\end{equation*}
When $x = \mathbf{T}^j\check{s}_{\beta,q}$, the entries of $\mathbf{A} \equiv \mathbf{A}(x)$ are given by
\[
    \begin{array}{ll}
    \mathbf{A}_{j, j} := a,  & \\
    \mathbf{A}_{k, k} := a', & k \ne j, \\    
    \mathbf{A}_{k, j} := -b, & k \ne j,
    \end{array}
\]
and all the other entries are zero---that is, 
\[
    \nabla g^{(k)}_{\beta}(\mathbf{T}^j \check{s}_{\beta,q})(s_1 - s_2)^{\top} =
    \begin{cases}
    a(s_1^{(j)} - s_2^{(j)}), &\text{if} \quad k = j,\\
    a'(s_1^{(k)} - s_2^{(k)}) - b(s_1^{(j)} - s_2^{(j)}), &\text{if} \quad k \ne j.
    \end{cases}
\]
\end{lemma}

We will defer the proof of Lemma~\ref{jacobian}, which relies on Lemma~\ref{g(x)=x} and algebraic manipulations, to Appendix~\ref{gradient_proof}. The next lemma bounds the Lipschitz constant of $g_\beta$ with respect to the $\ell_1$ norm for points that are sufficiently close to any equilibrium macrostate $x \in \mathfrak{S}_{\beta,q}$.

\begin{lemma} \label{L1contract}
Let $x \in \mathfrak{S}_{\beta,q}$, and $a > 0$ be as defined in Lemma~\ref{jacobian}. Define the positive constant $\theta(x, \beta, q)$ by
\[
    \theta(\hat{e}, \beta, q) := \frac{2 \beta}{q}
    \quad\text{and}\quad
    \theta(\mathbf{T}^j\check{s}_{\beta,q}, \beta, q) := a,\; j \in [q].
\]
Then for all $s_1, s_2 \in \mathcal{S}$ satisfying 
$\lVert s_1 - x \rVert_2 \leq r$ and $\lVert s_2 - x\rVert_2 \leq r$, we have
\begin{equation*}
    \frac{\lVert g_{\beta}(s_1) - g_{\beta}(s_2) \rVert_1}{\lVert s_1 - s_2 \rVert_1} \leq \theta(x, \beta, q) + O(r).
\end{equation*}
\end{lemma}

We will also defer the proof of Lemma~\ref{L1contract} to Appendix~\ref{L1contract_proof}. 
Now, we state a crucial technical condition that will be used to ensure that the Lipschitz constant of $g_{\beta}$ around any $x \in \mathfrak{S}_{\beta,q}$ is strictly less than $1$ for sufficiently small $r$, and to control the drift of the proportions chain $S(\tilde{\sigma}_t)$ induced by the restricted Glauber dynamics.

\begin{condition} \label{Cond:p}
$\theta(x, \beta, q) < 1$ and $\lambda(x, \beta, q) < 1$, where $\theta(x, \beta, q)$ is as defined in Lemma~\ref{L1contract} and $\lambda(x, \beta, q)$ is the maximum absolute eigenvalue of the symmetric part $(\mathbf{A} + \mathbf{A}^{\top}) / 2$ of the matrix $\mathbf{A} \equiv \mathbf{A}(x)$ defined in Lemma~\ref{jacobian}.
\end{condition}

Recalling that $\beta_s < \beta_c < q/2$ for $q \geq 3$, the following lemma shows that Condition~\ref{Cond:p} is satisfied under the assumptions on the inverse temperature parameter in Theorem~\ref{WDbound_high} (i.e.\ $\beta < \beta_s$) and Theorems~\ref{WDbound} and~\ref{mixingtime} (i.e.\ $\beta \geq \beta_s$). 
Note that Condition~\ref{Cond:p} is satisfied under more general assumptions; e.g.\ $\lambda(\hat{e}, \beta, q) = 2 \beta / q < 1$ holds for any $\beta < q/2$.

\begin{lemma} \label{expforcond}
Let $a, a', b > 0$ be as defined in Lemma~\ref{jacobian}, $\theta(x, \beta, q)$ be as defined in Lemma~\ref{L1contract}, and $\lambda(x, \beta, q)$ be as defined in Condition~\ref{Cond:p}. Then, one has $\lambda(\hat{e}, \beta, q) = 2 \beta / q$, and 
\begin{equation*}
    \lambda(\mathbf{T}^j\check{s}_{\beta,q}, \beta, q) = \frac{1}{2} \left( a + a' + \sqrt{(a - a')^2 + (q - 1) b^2} \right), \quad j \in [q].
\end{equation*}
Moreover, for all $q \geq 3$, the following statements hold:
\begin{enumerate}[label=\normalfont{(\arabic*)}]
    \item \label{expforcond_part1}
    If $x = \hat{e}$ and $\beta \leq \beta_c$, then $\theta(\hat{e}, \beta, q) = \lambda(x, \beta, q) = 2 \beta / q < 1$.
    \item \label{expforcond_part2}
    If $x = \mathbf{T}^j\check{s}_{\beta,q}$ and $\beta \geq \beta_c$, then $0 < b < a' < a = \theta(x, \beta, q) < \lambda(x, \beta, q) < 1$.
\end{enumerate}
\end{lemma}

The proof of Lemma~\ref{expforcond}, which involves fairly lengthy computations, will also be deferred to Appendix~\ref{expforcond_proof}. Parts of the proof are essentially embedded in the proofs of~\cite[Theorem~2.1, Proposition~2.2]{ellis1990limit}; however, we will provide a complete, self-contained proof.

The rest of this section focuses on proving Theorems~\ref{WDbound} and~\ref{mixingtime} in the low-temperature regime $\beta \geq \beta_s$, and is organised as follows. In Section~\ref{concenGD}, we show that with high probability, the Glauber dynamics restricted to $\widetilde{\Omega}(x, r)$ will stay in $\widetilde{\Omega}(x, \frac{4r}{5})$ and thus avoid the boundary for $O(N \log (N)^2)$ time. Next, in Section~\ref{coupling}, we specify a coupling of two copies of the restricted dynamics that is contracting according to~\eqref{eq:defn_contracting} as long as they stay within $\widetilde{\Omega}(x, \frac{4r}{5})$.
With these two results, the proof of Theorem~\ref{mixingtime} is rather straightforward and is given in Section~\ref{upperbound}. Finally,
we assemble the proof of Theorem~\ref{WDbound} in Section~\ref{pfWDB}, and briefly describe the modifications needed to prove Theorem~\ref{WDbound_high} in the high-temperature regime $\beta < \beta_s$ in Section~\ref{pfWDBhigh}.

\subsection{Concentration of the restricted Glauber dynamics} \label{concenGD}

Consider the Glauber dynamics for the Curie--Weiss--Potts model restricted to $\widetilde{\Omega}(x, r) = \left\{ \sigma \in \Omega: \lVert S(\sigma) - x \rVert_{2} \leq r \right\}$, denoted by $(\tilde{\sigma}_t)_{t \geq 0}$, with initial state $\sigma \in \widetilde{\Omega}(x, r)$. We shall denote the underlying probability measure by $\mathbb{P}^{x}_{\sigma}$,  the expectation by $\mathbb{E}^{x}_{\sigma}$, and the corresponding natural filtration by $\mathcal{F}_t$.

The main goal of this section is to show that if the restricted Glauber dynamics starts in $\widetilde{\Omega}(x, \frac{r}{5})$, then, with high probability, it will stay within $\widetilde{\Omega}(x, \frac{4r}{5})$ and avoid the boundary for $O(N \log (N)^2)$ time---this is chosen to be a sufficiently long period relative to the mixing time given in Theorem~\ref{mixingtime}. Moreover, if the chain starts outside $\widetilde{\Omega}(x, \frac{r}{5})$, then it will enter the region in $O(N)$ time.\footnote{Note that the factors $4/5$ and $1/5$ in the nested $\ell_2$ balls with radii of order $O(r)$ were chosen arbitrarily.} To be precise, we will prove the following:

\begin{lemma} \label{GlauberConcen}
For any $x \in \mathfrak{S}_{\beta,q}$, define the following stopping times for the restricted Glauber dynamics $(\tilde{\sigma}_t)_{t \geq 0}$ starting from $\sigma \in \widetilde{\Omega}(x, r)$:
\begin{equation*}
    \tau_{\mathrm{out}} := \inf\left\{ t \geq 0: \tilde{\sigma}_{t} \notin \widetilde{\Omega}\left( x, \frac{4r}{5} \right) \right\}
    \quad \text{and} \quad
    \tau_{\mathrm{in}} := \inf\left\{t\geq 0: \tilde{\sigma}_{t} \in \widetilde{\Omega}\left( x, \frac{r}{5} \right) \right\}.
\end{equation*}
Suppose that Condition~\ref{Cond:p} is satisfied and $r$ is sufficiently small. Then:
\begin{enumerate}[label=\normalfont{(\arabic*)}]
    \item \label{GlauberConcen_part1}
    There exists a constant $c > 0$ such that for all $\gamma > 0$ and $N$ large enough, 
    \begin{equation*}
        \mathbb{P}^{x}_{\sigma}\left( \tau_{\mathrm{out}} \leq \gamma N\log(N)^2 \right) \leq 2 \exp\left\{ -\frac{cN}{\gamma \log (N)^2} \right\},
    \end{equation*}
    for all $\sigma\in\widetilde{\Omega}(x, \frac{r}{5})$.
    
    \item \label{GlauberConcen_part2}
    There exists a constant $\gamma^* > 0$ and $c>0$ such that for all $N$ large enough, 
    \begin{equation*}
       \mathbb{P}^x_{\sigma}\left( \tau_{\mathrm{in}} > \gamma^* N \right) \leq e^{-cN},
    \end{equation*}
    for all $\sigma \in \widetilde{\Omega}(x, r)$.
\end{enumerate}
\end{lemma}

For any $x \in \mathfrak{S}_{\beta,q}$, let $(S_t)_{t \geq 0}$ with $S_t \coloneqq S(\tilde{\sigma}_t)$, considered as a row vector, denote the \emph{proportions chain} induced by the Glauber dynamics restricted to $\widetilde{\Omega}(x, r)$, which is a Markov chain on $\{ s \in \mathcal{S}: \lVert s - x \rVert_{2} \leq r\}$. We will also define the corresponding \emph{centred proportions chain} $(\widehat{S}_t)_{t \geq 0}$ by
\begin{equation}
    \widehat{S}_t \equiv \widehat{S}_t(x) := S_t - x.
\end{equation}
The following lemma describes the dynamics of the centred proportions chain $(\widehat{S}_t)_{t \geq 0}$ up to leading order by using Taylor series expansions for computing the drift.

\begin{lemma} \label{Xcontracting}
Let $(\tilde{\sigma}_t)_{t \geq 0}$ be the Glauber dynamics restricted to $\widetilde{\Omega}(x, r)$ for any $x \in \mathfrak{S}_{\beta, q}$, and $\widehat{S}_t$ be defined as above. Assuming that $\tilde{\sigma}_t$ is not at the boundary, then
\begin{equation*}
    \mathbb{E}^x_{\sigma}\left[ \widehat{S}_{t+1} - \widehat{S}_t \mid \mathcal{F}_t \right] = \frac{1}{N} \left( - \widehat{S}_t (\mathbf{I} - \mathbf{A}^{\top}) + O\left( \lVert \widehat{S}_t \rVert_2^2 \right) \right) + O(N^{-2}),
\end{equation*}
where $\mathbf{A} \equiv \mathbf{A}(x)$ is the matrix related to the Jacobian of $g_\beta$ at $x$ from Lemma~\ref{jacobian},
$\mathbf{I}$ is the $q \times q$ identity matrix, and the $O(\cdot)$ terms are understood to hold elementwise.
\end{lemma}

\begin{proof}
Let $e_k$ denote the standard basis vectors in $\reals^q$ (considered as row vectors, following our convention).
Recall that the coordinates of $S_t \in \mathcal{S}$ denote the proportions of vertices with each of the $q$ colours. According to the Glauber dynamics, the randomly selected vertex has colour $k \in [q]$ with probability $S_t^{(k)}$, and recoloured with a new colour $\ell \in [q]$ with probability $g_\beta^{(\ell)}\left( S_t - N^{-1} e_k \right)$.
Observe that by Taylor expansion of $g_\beta^{(\ell)}$ around $S_t$ with the mean-value form of the remainder, we have
\begin{equation} \label{eq:gbeta_taylor}
    g_\beta^{(\ell)}\left( S_t - N^{-1} e_k \right) = g_\beta^{(\ell)}\left( S_t \right) - N^{-1} \nabla g_\beta^{(\ell)}(\xi_k) e_k^{\top}
\end{equation}
for some vector $\xi_k$ in the line between $S_t$ and $S_t - N^{-1} e_k$. Furthermore, the derivatives of $g_\beta^{(\ell)}$ are bounded by constants (depending on $\beta$) on the compact probability simplex.
Thus, $S_t^{(\ell)}$ increases by $N^{-1}$ with probability
\begin{align*}
    \mathbb{P}^x_{\sigma}\left( S_{t+1}^{(\ell)} = S_t^{(\ell)} + N^{-1} \right)
    &= \sum_{\substack{k \in [q] \\ k \neq \ell}} g_{\beta}^{{(\ell)}}\left( S_t - N^{-1} e_{k} \right) S_t^{(k)} \\
    &= \sum_{k=1}^q g_{\beta}^{(\ell)}\left( S_t - N^{-1} e_{k} \right) S_t^{(k)} - g_{\beta}^{(\ell)}\left( S_t - N^{-1} e_{\ell} \right) S_t^{(\ell)} \\
    &= \left( 1 - S_t^{(\ell)} \right) \left( g_{\beta}^{(\ell)} \left( S_t \right) + O(N^{-1}) \right),
\end{align*}
where we used Taylor expansion~\eqref{eq:gbeta_taylor} and the fact that $\sum_{k=1}^q S_t^{(k)}=1$ for the last equality. Similarly, $S_t^{(\ell)}$ decreases by $N^{-1}$ with probability
\begin{align*}
    \mathbb{P}^x_{\sigma}\left( S_{t+1}^{(\ell)} = S_t^{(\ell)} - N^{-1} \right)
    &= \sum_{\substack{k \in [q] \\ k \neq \ell}} g_{\beta}^{(k)}\left( S_t - N^{-1} e_{\ell} \right) S_t^{(\ell)} \\
    & =\sum_{k=1}^q g_{\beta}^{(k)}\left( S_t - N^{-1} e_{\ell} \right) S_t^{(\ell)} - g_{\beta}^{(\ell)}\left( S_t - N^{-1} e_{\ell} \right) S_t^{(\ell)} \\
    &= S_t^{(\ell)} \left( 1 - g_{\beta}^{(\ell)}\left( S_t \right) + O(N^{-1}) \right).
\end{align*}
Therefore, for all $\ell = 1, \ldots, q$, one has
\begin{align*}
    \widehat{S}^{(\ell)}_{t+1} - \widehat{S}^{(\ell)}_t=
    \begin{cases}
        +\frac{1}{N} &\text{w.p.}\quad (1 - S^{(\ell)}_t) \left( g_{\beta}^{(\ell)}\left( S_t \right) + O(N^{-1}) \right),\\
        -\frac{1}{N} &\text{w.p.}\quad S^{(\ell)}_t\left( 1 - g_{\beta}^{(\ell)}\left( S_t \right) + O(N^{-1}) \right),
    \end{cases}
\end{align*}
and hence
\begin{equation} \label{Sdrift}
    \mathbb{E}^{x}_{\sigma}\left[ S_{t+1}^{(\ell)} - S_t^{(\ell)} \mid \mathcal{F}_t \right]
    = N^{-1} \left( -S_t^{(\ell)} + g_{\beta}^{(\ell)}(S_t) \right) + O(N^{-2}).
\end{equation}

First, we consider the case $x = \hat{e}$.
By Taylor expansion of $g_\beta^{(\ell)}$ around $\hat{e}$ up to higher order, recalling that $\widehat{S}_t = S_t - \hat{e}$, we obtain
\[
    g_\beta^{(\ell)}\left( S_t \right) = g_\beta^{(\ell)}\left( \hat{e} \right) - \nabla g_\beta^{(\ell)}(\hat{e}) \widehat{S}_t^{\top} + \sum_{j,k \in [q]} R_{j,k} \widehat{S}_t^{(j)} \widehat{S}_t^{(k)}.
\]
Note that the remainder terms $R_{j,k}$ can be uniformly bounded by a constant that depends on the higher-order derivatives of $g_\beta^{(\ell)}$, which are bounded on the compact probability simplex. Moreover, for all $j, k$, we can also bound $\widehat{S}_t^{(j)} \widehat{S}_t^{(k)} \leq \norm{\widehat{S}_t}_2^2$. 
Since $g_\beta^{(\ell)}\left( \hat{e} \right) = 1/q$ from Lemma~\ref{g(x)=x}, together with Lemma~\ref{jacobian}, we deduce that
\begin{align*}
    g_{\beta}^{(\ell)}\left( S_t \right)
    = \frac{1}{q} + \nabla g_{\beta}^{(\ell)}(\hat{e}) \widehat{S}^{\top}_t + O(\lVert \widehat{S}_t \rVert_2^2)
    = \frac{1}{q} + \frac{2\beta}{q} \widehat{S}^{(\ell)}_t + O(\lVert \widehat{S}_t \rVert_2^2).
\end{align*}
From~\eqref{Sdrift}, it follows that for all $\ell = 1, \ldots, q$,
\begin{equation*}
    \mathbb{E}^{\hat{e}}_{\sigma}\left[ S_{t+1}^{(\ell)} - S_t^{(\ell)} \mid \mathcal{F}_t \right]= N^{-1} \left( -\left( 1 - \frac{2\beta}{q} \right) \widehat{S}_t^{(\ell)} +O\left( \lVert \widehat{S}_t \rVert_2^2 \right) \right) + O(N^{-2}).
\end{equation*}

Next, we consider the case $x = \mathbf{T}^j\check{s}_{\beta,q}$, $j \in [q]$. By symmetry, it suffices to consider $x = \mathbf{T}^1\check{s}_{\beta,q} = \check{s}_{\beta,q}$. Similar to the case above, by using a Taylor series expansion of $g_{\beta}^{(\ell)}$ around $\check{s}_{\beta,q}$, recalling $\widehat{S}^t = S_t - \check{s}_{\beta,q}$ and Lemma~\ref{g(x)=x}, we deduce that
\begin{equation*} 
    g_{\beta}^{(\ell)}\left( S_t \right) ={\check{s}_{\beta,q}}^{(\ell)} + \nabla g_{\beta}^{(\ell)}(\check{s}_{\beta,q}) \widehat{S}^{\top}_t + O(\lVert \widehat{S}_t \rVert_2^2).
\end{equation*}
Furthermore, Lemma~\ref{jacobian} implies that
\begin{align*}
    \nabla g^{(1)}_{\beta}(\check{s}_{\beta,q}) \widehat{S}_t^{\top} &= a \widehat{S}_t^{(1)},\\
    \nabla g^{(\ell)}_{\beta}(\check{s}_{\beta,q}) \widehat{S}_t^{\top} &= -b \widehat{S}_t^{(1)} + a' \widehat{S}_t^{(\ell)}, \quad \ell = 2, \dots, q.
\end{align*}
Thus, from~\eqref{Sdrift}, it follows that
\begin{equation*}
    \mathbb{E}^{\check{s}_{\beta,q}}_{\sigma}\left[ \widehat{S}^{(1)}_{t+1} - \widehat{S}^{(1)}_t \mid \mathcal{F}_t \right] = N^{-1} \left( -(1 - a) \widehat{S}_t^{(1)} + O\left( \lVert \widehat{S}_t \rVert_2^2 \right) \right) + O(N^{-2});
\end{equation*}
and for $\ell = 2, \ldots, q$, it follows that
\begin{equation*}
    \mathbb{E}^{\check{s}_{\beta,q}}_{\sigma} \left[ \widehat{S}^{(\ell)}_{t+1} - \widehat{S}^{(\ell)}_t \mid \mathcal{F}_t \right] = N^{-1} \left( -(1 - a') \widehat{S}_t^{(\ell)} -b \widehat{S}_t^{(1)} + O\left( \lVert \widehat{S}_t\rVert_2^2 \right) \right) + O(N^{-2}).
\end{equation*}
Writing the above displayed expressions in matrix form completes the proof.
\end{proof} 

We will now prove Lemma~\ref{GlauberConcen} using the formula for the drift of the centred proportions vector given in Lemma~\ref{Xcontracting}. The proof relies on some standard hitting time estimates for supermartingale-like processes provided in~\cite[Lemma~2.1]{cuff2012}. For the reader's convenience, we restate the parts that we use in the following lemma.

\begin{lemma}[{\cite[Parts~(1) and~(2) of Lemma~2.1]{cuff2012}}] \label{cuff2.1}
Let $(D_t)_{t \geq 0}$ be a discrete-time process, adapted to $(\mathcal{F}_t)_{t > 0}$, with $D_0 = d_0 \in \mathbb{R}$ and underlying probability measure $\mathbb{P}_{d_0}$. Suppose that
\begin{itemize}
    \item There exists $\delta \geq 0$ such that $\mathbb{E}_{d_0}[D_{t+1} - D_{t} \mid \mathcal{F}_t] \leq -\delta$ on $\{D_t \geq 0\}$ for all $t \geq 0$.
    \item There exists $R \geq 0$ such that $\lvert  D_{t+1} - D_t \rvert \leq R$ for all $t \geq 0$.
\end{itemize}
Let $\tau_d^- := \inf\{ t : D_t \leq d \}$ and $\tau_d^+ := \inf\{ t : D_t > d \}$. Then the following hold:
\begin{enumerate}[label=\normalfont{(\arabic*)}]
    \item \label{cuff2.1_part1}
    If $d_0 \leq 0$ then for any $d_1 \geq R$ and $t_2 \geq 0$, 
    \begin{equation*}
        \mathbb{P}_{d_0}(\tau_{d_1}^+ \leq t_2) \leq  2 \exp\left\{ -\frac{(d_1 - R)^2}{8 t_2 R^2} \right\}.
    \end{equation*}

    \item \label{cuff2.1_part2}
    If $\delta > 0$ and $d_0 \geq 0$, then for any $t_1 \geq d_0 / \delta$,
    \begin{equation*}
        \mathbb{P}_{d_0}(\tau_{0}^- > t_1) \leq \exp\left\{ -\frac{(\delta t_1 - d_0)^2}{8 t_1 R^2} \right\}.
    \end{equation*}
\end{enumerate}
\end{lemma}

\begin{proof}[Proof of Lemma~\ref{GlauberConcen}]
Consider the case where the restricted Glauber dynamics $\tilde{\sigma}_t$ is not on the boundary, and thus all possible moves are allowed. If we write $\widehat{S}_{t+1} = \widehat{S}_{t} + \xi_{t+1}$, then
\begin{align}\label{drift}
    \mathbb{E}^x_{\sigma}\left[ \lVert \widehat{S}_{t+1} \rVert_2^2 \mid \mathcal{F}_t \right]
    &= \mathbb{E}^x_{\sigma}\left[ \lVert \widehat{S}_{t} \rVert_2^2 + \lVert \xi_{t+1} \rVert_2^2 + 2\langle \xi_{t+1}, \widehat{S}_t \rangle \mid \mathcal{F}_t \right] \nonumber\\
    &= \lVert \widehat{S}_{t} \rVert_2^2 + \mathbb{E}^x_{\sigma}\left[ \lVert \xi_{t+1} \rVert_2^2  \mid \mathcal{F}_t \right] + 2\langle \mathbb{E}^x_{\sigma}\left[ \xi_{t+1} \mid \mathcal{F}_t \right], \widehat{S}_t \rangle.
\end{align}
First, we have the estimate
\begin{equation}\label{eq:xi_bound}
    \norm{\xi_{t+1}}_2^2 = \norm{\widehat{S}_{t+1} - \widehat{S}_t}_2^2 = \norm{S_{t+1} - S_t}_2^2 \leq 2 N^{-2}.
\end{equation}
Next, if $\mathbf{A} \equiv \mathbf{A}(x)$ denotes the matrix related to the Jacobian of $g_\beta$ at $x$ from Lemma~\ref{jacobian}, then Lemma~\ref{Xcontracting} implies that
\begin{equation*}
    \mathbb{E}^x_{\sigma}\left[ \xi_{t+1} \mid \mathcal{F}_t \right]
    = \mathbb{E}^x_{\sigma}\left[ \widehat{S}_{t+1} - \widehat{S}_t \mid \mathcal{F}_t \right] = \frac{1}{N} \left( -\widehat{S}_t (\mathbf{I} - \mathbf{A}^{\top}) + O\left( \lVert \widehat{S}_t \rVert_2^2 \right) \right) + O(N^{-2}),
\end{equation*}
We claim that
\begin{equation} \label{claim:contracting}
    \langle \mathbb{E}^x_{\sigma}\left[ \xi_{t+1} \mid \mathcal{F}_t \right], \widehat{S}_t \rangle
    \leq -\frac{(1 - \lambda(x,\beta, q) +O(\lVert \widehat{S}_t \rVert_2))}{N} \lVert \widehat{S}_{t} \rVert_2^2 + O(N^{-2}),
\end{equation}
where $\lambda(x, \beta, q)$ is the maximum absolute eigenvalue of the symmetric matrix $(\mathbf{A} + \mathbf{A}^{\top}) / 2$, which satisfies $\lambda(x, \beta, q) \in (0, 1)$ assuming Condition~\ref{Cond:p} holds. The claim~\eqref{claim:contracting} follows from the following inequality involving the quadratic form of $\mathbf{A}$: for any $s \in \reals^q$,
\[
    \langle s \mathbf{A}^{\top}, s \rangle
    = s \mathbf{A} s^{\top} 
    = s \left( \frac{\mathbf{A} + \mathbf{A}^{\top}}{2} \right) s^{\top} 
    \leq \lambda(x, \beta, q) \norm{s}_2^2.
\]
The second equality follows from the decomposition $\mathbf{A} = (\mathbf{A} + \mathbf{A}^{\top}) / 2 + (\mathbf{A} - \mathbf{A}^{\top}) / 2$ of $\mathbf{A}$ into its symmetric and skew-symmetric parts, and using the fact that the quadratic form corresponding to the skew-symmetric part is identically zero. The inequality follows from diagonalising the real symmetric matrix $(\mathbf{A} + \mathbf{A}^{\top}) / 2$ to perform an orthogonal change of basis, and bounding the corresponding eigenvalues by $\lambda(x, \beta, q)$.

Hence, by plugging the bounds~\eqref{eq:xi_bound} and~\eqref{claim:contracting} back into~\eqref{drift} and using the fact that $\lVert \widehat{S}_t \rVert_2 \leq r$, it follows that if the restricted dynamics is not on the boundary, then
\begin{equation} \label{contracting}
    \mathbb{E}^x_{\sigma}\left[ \lVert \widehat{S}_{t+1} \rVert_2^2 - \lVert \widehat{S}_{t} \rVert_2^2 \mid \mathcal{F}_t \right]
    \leq -\frac{2(1 - \lambda(x, \beta, q) + O(r))}{N} \lVert \widehat{S}_t \rVert_{2}^2 + O(N^{-2}).
\end{equation}
Under Condition~\ref{Cond:p}, the right-hand side of~\eqref{contracting} is negative for large enough $N$ and small enough $r$.

Now consider the case where the restricted dynamics is on the boundary. Since there is a possible move that leads to rejection, we must have $\lVert \widehat{S}(\sigma_t) \rVert_2 > r - \sqrt{2} N^{-1}$. Moreover, since each move that increases $\lVert \widehat{S}_{t} \rVert_2^2$ is rejected, we may upper bound the change in $\lVert \widehat{S}_{t} \rVert_2^2$ by using the bound~\eqref{contracting} from above, where all possible moves are accepted.
Therefore, we have
\begin{equation}\label{contractingboundary}
    \mathbb{E}^x_{\sigma}\left[ \lVert \widehat{S}_{t+1} \rVert_2^2 - \lVert \widehat{S}_{t} \rVert_2^2 \mid \mathcal{F}_t \right] \leq -\frac{2(1 - \lambda(x, \beta, q) + O(r))}{N} \left( r - \frac{\sqrt{2}}{N} \right)^2 + O(N^{-2})
\end{equation}

Now, by the triangle inequality and~\eqref{eq:xi_bound}, one has
\begin{equation*}
    \lVert \widehat{S}_{t+1} \rVert_2^2 - \lVert \widehat{S}_{t} \rVert_2^2 = \left(\lVert \widehat{S}_{t+1} \rVert_2 + \lVert \widehat{S}_{t} \rVert_2 \right) \left( \lVert \widehat{S}_{t+1} \rVert_2 - \lVert \widehat{S}_{t} \rVert_2 \right)
    \leq 2 r \lVert \xi_{t+1} \rVert_2
    < 4 r N^{-1}.
\end{equation*}
Thus, from the discussion above, we know that the increments of $\lVert \widehat{S}_t \rVert_{2}^2$ are negative and satisfy the conditions of Lemma~\ref{cuff2.1} with $\delta \geq 0$ and $R = 4 r N^{-1}$.

For Part~\ref{GlauberConcen_part1}, given any initial state $\sigma \in \widetilde{\Omega}(x, \frac{r}{5})$ with $\norm{\widehat{S}_0}_2 \leq r/5$, note that the restricted dynamics is separated from the boundary before $\tau_{\mathrm{out}}$. Furthermore, observe that if $\lVert \widehat{S}_t \rVert_2 - \lVert \widehat{S}_0 \rVert_2 > 3r / 5$, then $\lVert \widehat{S}_t \rVert_2^2 - \lVert \widehat{S}_0 \rVert_2^2 > \left( 3r / 5 \right)^2$. 
Therefore, applying Part~\ref{cuff2.1_part1} of Lemma~\ref{cuff2.1} to the process $D_t = \lVert \widehat{S}_t \rVert_{2}^2 - \lVert \widehat{S}_0 \rVert_{2}^2$, with $t_2 = \gamma N\log(N)^2$ and $d_1 = (r/5)^2$, shows that
\begin{equation*}
    \mathbb{P}^x_{\sigma}\left( \tau_{\rm{out}} \leq \gamma N \log(N)^2 \right)
    \leq \mathbb{P}_{d_0}\left( \tau_{d_1}^{+} \leq \gamma N \log(N)^2 \right)
    \leq 2 \exp\left\{ -\frac{cN}{\gamma \log(N)^2} \right\},
\end{equation*}
for some constant $c > 0$ depending on $r$, which completes the proof of Part~\ref{GlauberConcen_part1}.

For Part~\ref{GlauberConcen_part2}, given any initial state $\sigma \in \widetilde{\Omega}(x, r)$, note that before $\tau_{\mathrm{in}}$, one has $\lVert \widehat{S}_t \rVert_{2}^2 > r/5$. Thus, from~\eqref{contracting} and~\eqref{contractingboundary}, we deduce that during this period, the increments of $\lVert \widehat{S}_t \rVert_{2}^2$ satisfy the conditions of Lemma~\ref{cuff2.1} with $\delta = C' N^{-1} > 0$ for some positive constant $C'$ depending on $r, \beta, q$. Therefore, choosing $\gamma^* = 4 / C'$
and applying Part~\ref{cuff2.1_part2} of Lemma~\ref{cuff2.1} to the process $D_t = \lVert \widehat{S}_t \rVert_{2}^2 - (r/5)^2$ with $t_1 = \gamma^* N$ yields
\begin{equation*}
    \mathbb{P}^x_{\sigma}\left( \tau_{\rm{in}} > \gamma^* N \right)
    = \mathbb{P}_{d_0}\left( \tau_{0}^{-} > \gamma^* N \right)
    \leq e^{-cN},
\end{equation*}
for some $c > 0$ depending on $\gamma^*$ and $r$, which completes the proof of Part~\ref{GlauberConcen_part2}.
\end{proof}

\subsection{A contracting coupling of the restricted Glauber dynamics} \label{coupling}

Let $(W_{t}, Z_{t})_{t \geq 0}$ be a coupling for two copies of the Glauber dynamics for the Curie--Weiss--Potts model restricted to $\widetilde{\Omega}(x, r)$, with transitions to be specified later. We write $\mathbb{P}^{x}_{\sigma, \tau}$ for the underlying probability measure of the coupling with initial states $\sigma, \tau \in \widetilde{\Omega}(x, r)$, $\mathbb{E}^{x}_{\sigma, \tau}$ for the expectation, and denote the natural filtration by $\mathcal{F}_t$.

The main result of this section is the following, which states that there exists a coupling such that the Hamming distance between the two chains is decreasing on average, as long as they stay within the good region $\widetilde{\Omega}(x, \mbox{$\frac{4r}{5}$})$ and thus avoid the boundary.

\begin{lemma} \label{GDcontrcting}
For any $\gamma > 0$, define the event 
\begin{equation}\label{def:E}
    B := \left\{ \text{For all } t \leq \gamma N \log(N)^2,\; W_{t} \in \widetilde{\Omega}(x, \mbox{$\frac{4r}{5}$}) \text{ and } Z_{t} \in \widetilde{\Omega}(x, \mbox{$\frac{4r}{5}$}) \right\}.
\end{equation}
Suppose that Condition~\ref{Cond:p} is satisfied and $r$ is sufficiently small. Then for all $\sigma, \tau \in \widetilde{\Omega}(x, \frac{r}{5})$, there exists a coupling $(W_{t}, Z_{t})_{t \geq 0}$ of the Glauber dynamics restricted to $\widetilde{\Omega}(x, r)$, starting from $\sigma$ and $\tau$, such that if $r$ is sufficiently small and $N$ is large enough, then for all $t \leq \gamma N \log(N)^2$,
\begin{equation} \label{glaubercontract}
    \mathbb{E}^{x}_{\sigma, \tau} \left[ d_H(W_{t+1}, Z_{t+1}) \indicator{B} \right] \leq \left( 1 - \frac{1 - \theta(x, \beta, q)}{2N} \right)^{t} d_H(\sigma, \tau) \, \mathbb{P}^{x}_{\sigma, \tau}(B),
\end{equation}
where $\theta(x, \beta, q) < 1$ is as defined in Lemma~\ref{L1contract}.
\end{lemma}

\begin{proof}
We first specify the coupling $(W_{t}, Z_{t})_{t \geq 0}$ by defining its transitions. Let $(W_{0}, Z_{0}) = (\sigma, \tau)$. At time $t+1$, choose a vertex $v \in [N]$ uniformly at random and draw new colours $i, j \in [q]$ according to an optimal coupling of $g_{\beta}\left( S(W_t) - N^{-1} e_{W_t(v)} \right)$ and $g_{\beta}\left( S(Z_t) - N^{-1} e_{Z_t(v)} \right)$, which are the conditional spin distributions at $v$ of $W_t$ and $Z_t$. Next, propose two new configurations $\sigma', \tau' \in \Omega$ by recolouring the vertex $v$ of $W_t$ and $Z_t$ with $i$ and $j$ respectively, i.e.
\begin{align*}
    \sigma'(u) =
    \begin{cases}
        W_{t}(u), \quad &\text{if}\quad u \neq v,\\
        i, \quad &\text{if}\quad u = v,
    \end{cases}
    \quad\text{and}\quad
    \tau'(u) =
    \begin{cases}
        Z_{t}(u), \quad &\text{if}\quad u \neq v,\\
        j, \quad &\text{if}\quad u = v.
    \end{cases}
\end{align*}
If $\sigma' \in \widetilde{\Omega}(x, r)$, then set $W_{t+1} = \sigma'$.  Otherwise, reject the move and set $W_{t+1} = W_t$. Similarly, set $Z_{t+1} = \tau'$ if $\tau' \in \widetilde{\Omega}(x, r)$, and otherwise set $Z_{t+1} = Z_t$. Note that once the two chains have coalesced (i.e.\ they are equal at some time), then they will continue to move together afterwards.

Given that both proposed moves are accepted, the probability that the colour at $v$ differs between the two chains is equal to
\begin{align*}
    \rho &\coloneqq \dtv\left( g_{\beta}\left( S(W_t) - N^{-1} e_{W_t(v)} \right), g_{\beta}\left( S(Z_t) - N^{-1} e_{Z_t(v)} \right) \right)\\
    &\phantom{:}= \frac{1}{2} \left\lVert g_{\beta}\left( S(W_t) - N^{-1} e_{W_t(v)}\right) - g_{\beta}\left( S(Z_t) - N^{-1} e_{Z_t(v)} \right) \right\rVert_1.
\end{align*}
If the two chains have the same colour at $v$ at time $t$, say $W_t(v) = Z_{t}(v) = \ell \in [q]$, then by Taylor series expansion of the function $g^{(k)}_{\beta}$ around $S(W_{t})$, we have
\begin{align*}
    &g^{(k)}_{\beta}\left( S(W_{t}) - N^{-1} e_{\ell} \right) - g^{(k)}_{\beta}\left( S(Z_{t}) - N^{-1}e_{\ell} \right)\\
    &\qquad= g^{(k)}_{\beta}\left( S(W_{t})\right) - g^{(k)}_{\beta}\left( S(Z_{t}) \right) + N^{-1} \left( \nabla g_{\beta}^{(k)}( S(W_{t})) - \nabla g_{\beta}^{(k)}(S(Z_{t})) \right) e_{\ell}^{\top} + O(N^{-2}).
\end{align*}
By further Taylor series expansion of $\nabla g_{\beta}^{(k)}$ around $x$, and using the fact that both $\norm{S(W_t) - x}_1$ and $\norm{S(Z_t) - x}_1$ are less than $r$, we have
\begin{align*}
    g^{(k)}_{\beta}\left( S(W_{t}) - N^{-1} e_{\ell} \right) - g^{(k)}_{\beta}\left( S(Z_{t}) - N^{-1} e_{\ell} \right)
    = g^{(k)}_{\beta}\left( S(W_{t}) \right) - g^{(k)}_{\beta}( S(Z_{t}) ) + O(r) N^{-1} + O(N^{-2}).
\end{align*}
Here, the notation $O(\cdot)$ hides constants related to higher order derivatives of $g_\beta^{(k)}$.
Hence, by summing over the $q$ colours and using the triangle inequality, we obtain the following bound for the probability that the two chains disagree after an update at $v$:
\begin{align*}
    \rho \leq \frac{1}{2} \left\lVert g_{\beta}\left( S(W_{t}) \right) - g_{\beta}\left( S(Z_{t})\right)\right\rVert_1 + O(r) N^{-1} + O(N^{-2}).
\end{align*}

On the other hand, if $W_t(v) \neq Z_{t}(v)$, we have a cruder bound:
\begin{align*}
    \rho \leq \frac{1}{2} \left\lVert g_{\beta}\left( S(W_{t}) \right) - g_{\beta}\left( S(Z_{t} ) \right)\right\rVert_1 + O(N^{-1}).
\end{align*}

Observe that on the event $B$, both chains avoid the boundary, and thus none of the proposed moves are rejected before $\gamma N \log(N)^2$ time. Hence, during this period, the distance between the two copies will increase by one if $v$ is selected to be one of the vertices where $W_t$ and $Z_t$ agree, and $i \ne j$. On the other hand, the distance will decrease by one if $v$ is selected to be one of the vertices where $W_{t}$ and $Z_t$ disagree, and $i = j$. Otherwise, the distance does not change. Therefore, we have
\begin{align*}
    &\mathbb{E}^{x}_{\sigma, \tau}\left[ d_H(W_{t+1}, Z_{t+1}) \indicator{B} - d_H(W_{t}, Z_{t}) \indicator{B} \mid \mathcal{F}_t \right]\\
    &\qquad\leq \indicator{B} \left( 1 - \frac{d_H(W_{t}, Z_{t})}{N} \right) \left( \frac{1}{2} \left\lVert g_{\beta}\left( S(W_{t}) \right) - g_{\beta}\left( S(Z_{t}) \right) \right\rVert_1 + O(r) N^{-1} + O(N^{-2}) \right)\\
    &\qquad\qquad - \indicator{B} \frac{d_H(W_{t}, Z_{t})}{N} \left( 1 - \frac{1}{2} \left\lVert g_{\beta}\left( S(W_{t}) \right) - g_{\beta}\left( S(Z_{t}) \right) \right\rVert_1 + O(N^{-1}) \right)\\
    &\qquad \leq -\indicator{B} \left( \frac{d_H(W_{t}, Z_{t})}{N} \left( 1 - \frac{\left\lVert g_{\beta} \left( S(W_{t}) \right) - g_{\beta}(S(Z_{t})) \right\rVert_1}{2 d_H(W_{t}, Z_{t}) N^{-1}} \right) + O(r) N^{-1}  + O(N^{-2})\right).
\end{align*}
Note that $\lVert S(\sigma') - S(\tau') \rVert_1 \leq 2 d_H(\sigma', \tau') N^{-1}$ for any pair $\sigma', \tau' \in \Omega$, since each location where they differ contributes at most $2N^{-1}$ to $\lVert S(\sigma') - S(\tau') \rVert_1$. Since
Lemma~\ref{L1contract} implies that $g_{\beta}$ is $(\theta(x, \beta, q) + O(r))$-Lipschitz around $x$ with $\theta(x, \beta, q) < 1$ assuming Condition~\ref{Cond:p} holds, we have
\begin{align*}
    &\mathbb{E}^{x}_{\sigma, \tau}\left[ d_H(W_{t+1}, Z_{t+1}) \indicator{B} - d_H(W_{t}, Z_{t}) \indicator{B} \mid \mathcal{F}_t \right]\\
    &\qquad\leq -\indicator{B} \left( \frac{d_H(W_{t}, Z_{t})}{N} \left( 1 - \frac{\left\lVert g_{\beta}\left( S(W_{t}) \right) - g_{\beta}\left( S(Z_{t}) \right) \right\rVert_1}{\left\lVert S(W_t) - S(Z_t) \right\rVert_1} \right) + O(r) N^{-1} + O(N^{-2}) \right) \\
    &\qquad\leq -\frac{(1 - \theta(x, \beta, q) + O(r)) d_H(W_{t}, Z_{t}) + O(r)+ O(N^{-1})}{N}  \indicator{B} \\
    &\qquad\leq -\frac{1 - \theta(x, \beta, q)}{2N} d_H(W_{t}, Z_{t}) \indicator{B},
\end{align*}
for sufficiently small $r$ and large $N$. By taking expectations on both sides, we obtain
\begin{equation*}
    \mathbb{E}^{x}_{\sigma, \tau}\left[ d_H(W_{t+1}, Z_{t+1}) \indicator{B} \right]
    \leq \left( 1 - \frac{1 - \theta(x, \beta, q)}{2N} \right) \mathbb{E}^{x}_{\sigma, \tau}\left[ d_H(W_{t}, Z_{t}) \indicator{B} \right].
\end{equation*}
The proof is then completed by iterating this bound.
\end{proof}

\subsection{Mixing time results for the restricted Glauber dynamics} \label{upperbound}

This section is devoted to the proof of Theorem~\ref{mixingtime}.
Given any coupling $(W_{t}, Z_{t})_{t \geq 0}$ of two copies of the restricted Glauber dynamics, starting from $W_0 = \sigma$ and $Z_0 = \tau$, we define the associated \emph{coalescence time} by
\begin{equation}
    \tau_{\mathrm{couple}} := \inf\{ t: W_{t} = Z_{t} \text{ for all } s \geq t\},
\end{equation}
By standard results for the coupling method (see~\cite[Corollary 5.5]{levin2017markov}), the mixing time of a Markov chain is bounded by the tail probabilities of any coupling of the chain in the worst case over all pairs of initial configurations $\sigma, \tau \in \widetilde{\Omega}(x, r)$.
Therefore, in order to prove Theorem~\ref{mixingtime}, it suffices to prove the following lemma:

\begin{lemma} \label{col}
Suppose that Condition~\ref{Cond:p} is satisfied and $r$ is sufficiently small.
Then for all $\sigma, \tau \in \widetilde{\Omega}(x, r)$, there exists a coupling $(W_t, Z_t)_{t \geq 0}$ of the Glauber dynamics restricted to $\widetilde{\Omega}(x, r)$, starting from $\sigma$ and $\tau$, such that
\begin{equation*}
    \limsup_{N \to \infty} \max_{\sigma, \tau \in \widetilde{\Omega}(x, r)} \mathbb{P}^{x}_{\sigma, \tau} \left( \tau_{\mathrm{couple}} > \frac{2}{1 - \theta(x, \beta, q)}N\log(N) + \gamma^* N + \alpha N \right) \to 0 \quad \text{as} \quad \alpha \to \infty,
\end{equation*}
where $\theta(x, \beta, q) < 1$ is as defined in Lemma~\ref{L1contract} and $\gamma^* > 0$ is the constant from Part~\ref{GlauberConcen_part2} of Lemma~\ref{GlauberConcen}.
\end{lemma}

To prove Lemma~\ref{col}, we shall use the following coupling of the restricted Glauber dynamics, starting from $\sigma, \tau \in \widetilde{\Omega}(x, r)$. In the first phase, we run the two copies independently until they are both inside $\widetilde{\Omega}(x, \frac{r}{5})$. In the second phase, the two chains are then coupled using the coupling previously described in Section~\ref{coupling}. If at any time the two copies coalesce, the two chains will move together afterwards.

The proof then proceeds in two steps. First, we show that if both copies of the restricted dynamics start in $\widetilde{\Omega}(x, \frac{r}{5})$, then the coalescence time is of order $O(N \log(N))$ with high probability.
Note that the upcoming lemma differs from Lemma~\ref{col} in the set of starting points considered.

\begin{lemma} \label{colgoodset}
Consider the coupling of the restricted Glauber dynamics described above, starting from any $\sigma, \tau\in \widetilde{\Omega}(x, \frac{r}{5})$. If Condition~\ref{Cond:p} is satisfied and $r$ is sufficiently small, then for all $\alpha > 0$,
\begin{equation*}
    \mathbb{P}^{x}_{\sigma, \tau}\left( \tau_{\mathrm{couple}} > \frac{2}{1 - \theta(x, \beta, q)} N \log(N) + \alpha N \right) 
    \leq \exp\left\{ -\frac{1 - \theta(x, \beta, q)}{2}\alpha \right\} + 4 N \exp \left\{ -\frac{cN}{\gamma \log(N)^2} \right\},
\end{equation*}
where $\gamma = 2 / (1 - \theta(x, \beta, q)) + \alpha$, $\theta(x, \beta, q) < 1$ is as defined in Lemma~\ref{L1contract}, and $c > 0$ is the constant from Part~\ref{GlauberConcen_part1} of Lemma~\ref{GlauberConcen}.
\end{lemma}

\begin{proof}
Let $B$ be the event defined in Lemma~\ref{GDcontrcting} with $\gamma = 2 / (1 - \theta(x, \beta, q)) + \alpha$. Since the Hamming distance between any two configurations is bounded by $N$, we have
\begin{align} \label{onE}
    \mathbb{E}^{x}_{\sigma, \tau}\left[ d_H(W_{t}, Z_{t}) \right]
    &= \mathbb{E}^{x}_{\sigma, \tau}\left[ d_H(W_{t}, Z_{t}) \indicator{B} \right] + \mathbb{E}^{x}_{\sigma, \tau} \left[ d_H(W_{t}, Z_{t}) \indicator{B^c} \right] \nonumber\\
    &\leq \mathbb{E}^{x}_{\sigma, \tau}\left[ d_H(W_{t}, Z_{t}) \indicator{B} \right] + N \mathbb{P}^{x}_{\sigma, \tau}(B^c).
\end{align}
By applying Part~\ref{GlauberConcen_part1} of Lemma~\ref{GlauberConcen} to each copy of the restricted Glauber dynamics and using a union bound, we have
\begin{equation*}
    \mathbb{P}^{x}_{\sigma, \tau}(B) > 1 - 4 \exp\left\{ -\frac{cN}{\gamma \log(N)^2} \right\}
\end{equation*}
for all initial states $\sigma, \tau \in \widetilde{\Omega}(x, \frac{r}{5})$.
It then follows from Lemma~\ref{GDcontrcting} that for all $t \leq \gamma N \log(N)^2$,
\begin{align}
    \mathbb{E}^{x}_{\sigma, \tau}\left[ d_H(W_{t}, Z_{t})\right]
    &\leq \left(1 - \frac{1 - \theta(x, \beta, q)}{2N} \right)^{t} d_H(\sigma, \tau) \mathbb{P}^{x}_{\sigma, \tau}(B) + N \mathbb{P}^{x}_{\sigma, \tau}(B^c) \nonumber\\
    &\leq \exp\left\{ -\frac{1 - \theta(x, \beta, q)}{2N} t \right\} d_H(\sigma, \tau) + 4 N \exp\left\{ -\frac{cN}{\gamma\log(N)^2} \right\}. \label{eq:colgoodset_pf1}
\end{align}
Since $\mathbb{P}^{x}_{\sigma, \tau}\left( \tau_{\mathrm{couple}} > t \right) = \mathbb{P}^{x}_{\sigma, \tau}\left( d_H\left(W_{t}, Z_{t} \right) \geq 1 \right)$, an application of Markov inequality yields
\begin{align*}
    \mathbb{P}^{x}_{\sigma, \tau}\left( \tau_{\mathrm{couple}} > t \right)
    \leq \mathbb{E}^{x}_{\sigma, \tau}\left[ d_H\left(W_{t}, Z_{t}\right) \right]
    \leq N \exp\left\{ -\frac{1 - \theta(x, \beta, q)}{2N} t \right\} + 4 N \exp\left\{ -\frac{cN}{\gamma \log(N)^2} \right\}.
\end{align*}
By substituting $t = (2 / (1 - \theta(x, \beta, q)) N \log(N) + \alpha N$ above, we obtain the desired bound.
\end{proof}

Since both chains will enter $\widetilde{\Omega}(x, \frac{r}{5})$ in $O(N)$ time, starting from any pair of configurations in $\widetilde{\Omega}(x, r)$, we can now prove Lemma~\ref{col} using Lemma~\ref{colgoodset}.

\begin{proof}[Proof of Lemma~\ref{col}]
For all $\sigma, \tau \in \widetilde{\Omega}(x, r)$, it follows from applying Part~\ref{GlauberConcen_part2} of Lemma~\ref{GlauberConcen} to each copy of the restricted dynamics, running independently in the first phase of the coupling, and then using a union bound, that
\begin{align*}
    &\mathbb{P}^{x}_{\sigma, \tau}\left( \tau_{\mathrm{couple}} > \mbox{$\frac{2}{1 - \theta(x, \beta, q)}$} N \log(N) + \gamma^*N + \alpha N \right)\\
    &\quad \leq \mathbb{P}^{x}_{\sigma, \tau}\left( \tau_{\mathrm{couple}} > \mbox{$\frac{2}{1 - \theta(x, \beta, q)}$} N \log(N) + \gamma^*N +\alpha N,\, W_{\gamma^* N} \in \widetilde{\Omega}(x, \mbox{$\frac{r}{5}$}),\, Z_{\gamma^* N} \in \widetilde{\Omega}(x, \mbox{$\frac{r}{5}$}) \right)\\
    &\quad\qquad + \mathbb{P}^{x}_{\sigma, \tau}\left(\{W_{\gamma^* N} \notin \widetilde{\Omega}(x, \mbox{$\frac{r}{5}$})\} \cup \{Z_{\gamma^* N} \notin \widetilde{\Omega}(x, \mbox{$\frac{r}{5}$})\} \right)\\
    &\quad \leq \mathbb{P}^{x}_{\sigma, \tau}\left( \tau_{\mathrm{couple}} > \mbox{$\frac{2}{1 - \theta(x, \beta, q)}$} N \log(N) + \gamma^* N + \alpha N,\, W_{\gamma^* N} \in \widetilde{\Omega}(x, \mbox{$\frac{r}{5}$}),\, Z_{\gamma^* N} \in \widetilde{\Omega}(x, \mbox{$\frac{r}{5}$}) \right) + o(1).
\end{align*}
By using Lemma~\ref{colgoodset}, noting that the upper bound is independent of the starting states, and the strong Markov property, it follows that for generic $\sigma', \tau' \in \widetilde{\Omega}(x, \frac{r}{5})$,
\begin{align*}
     \mathbb{P}^{x}_{\sigma, \tau}&\left( \tau_{\mathrm{couple}} > \mbox{$\frac{2}{1 - \theta(x, \beta, q)}$} N \log(N) + \gamma^* N + \alpha N,\, W_{\gamma^* N} \in \widetilde{\Omega}(x, \mbox{$\frac{r}{5}$}),\, Z_{\gamma^* N} \in \widetilde{\Omega}(x, \mbox{$\frac{r}{5}$})  \right)\\
     &\leq \mathbb{P}^{x}_{\sigma', \tau'}\left( \tau_{\mathrm{couple}} > \mbox{$\frac{2}{1 - \theta(x, \beta, q)}$} N \log(N) + \alpha N \right)\\
     &\leq \exp\left\{ -\frac{1 - \theta(x, \beta, q)}{2} \alpha \right\} + 4 N \exp\left\{ -\frac{cN}{\gamma\log(N)^2} \right\}.
\end{align*}
By combining the last two displays,
we obtain the bound
\begin{equation*}
    \mathbb{P}^{x}_{\sigma, \tau} \left(\tau_{\mathrm{couple}} > \frac{2}{1 - \theta(x, \beta, q)} N \log(N) + \gamma^* N + \alpha N \right)
    \leq \exp\left\{ -\frac{1 - \theta(x, \beta, q)}{2} \alpha \right\} + o(1),
\end{equation*}
as $N \to \infty$, which completes the proof.
\end{proof}

\subsection{Approximation result in low-temperature regime} \label{pfWDB}

In this section, we assemble the results from the previous sections to prove our main result when $\beta \geq \beta_s$.

\begin{proof}[Proof of Theorem~\ref{WDbound}]
First, we note that Condition~\ref{Cond:p} holds under the assumption $\beta \geq \beta_s$ by Lemma~\ref{expforcond}. Let $\widetilde{X}_t$ be the Glauber dynamics restricted to $\widetilde{\Omega}(x, r)$ with stationary distribution $\tilde{\mu}$.
We will derive an analogue of Lemma~\ref{lem:continuous_glauber_generator} for the restricted Glauber dynamics.
Analogous to~\eqref{eq:glauber_generator}, the generator $\mathcal{A}_{\tilde{\mu}} = \widetilde{P} - I$ for the continuous-time restricted dynamics $\widetilde{P}$ takes the following form: for $\sigma \in \widetilde{\Omega}(x, r)$, using the fact that $\tilde{\mu}_v( k \mid \sigma ) = \mu_v( k \mid \sigma ) \indicator{\sigma^{(v,k)} \in \widetilde{\Omega}(x, r)}$ and $\sum_{k \in [q]} \mu_v( k \mid \sigma ) = 1$, one has
\begin{align}
    \mathcal{A}_{\tilde{\mu}} f(\sigma)
    &= \frac{1}{N} \sum_{v \in V} \left[ \sum_{k \in [q]} \tilde{\mu}_v( k \mid \sigma ) f(\sigma^{(v, k)}) + \left( 1 - \sum_{k \in [q]} \tilde{\mu}_v(k \mid \sigma) \right) f(\sigma) \right] - f(\sigma) \nonumber\\
    &= \frac{1}{N} \sum_{v \in V} \left[ \sum_{k \in [q]} \mu_v( k \mid \sigma ) (f(\sigma^{(v, k)}) - f(\sigma)) \indicator{\sigma^{(v,k)} \in \widetilde{\Omega}(x, r)} \right]. \label{eq:restricted_glauber_generator}
\end{align}
The generator $\mathcal{A}_{\tilde{\nu}}$ for the restricted Glauber dynamics for $\tilde{\nu}$ takes the same form with $\mu_v( \cdot \mid \sigma )$ replaced by $\nu_v( \cdot \mid \sigma )$. By inserting~\eqref{eq:restricted_glauber_generator} into the Stein's method bound $|\mathbb{E} h(\widetilde{X}) - \mathbb{E} h(\widetilde{Y})| = |\mathbb{E} \mathcal{A}_{\tilde{\mu}} f_h(\widetilde{Y}) - \mathbb{E} \mathcal{A}_{\tilde{\nu}} f_h(\widetilde{Y})|$ from~\eqref{eq:stein_eqn_comparison} with $\widetilde{X} \sim \tilde{\mu}$ and $\widetilde{Y} \sim \tilde{\nu}$, recalling that $\mu_v(k \mid \sigma) = g_{\beta}^{(k)}(\sigma)$ and $\nu_v(k \mid \sigma) = x^{(k)}$, we obtain the following bound for any $h: \widetilde{\Omega}(x, r) \to \reals$ in terms of the solution $f_h$ to the Stein equation~\eqref{eq:stein_eqn}:
\begin{equation} \label{steinbound}
\begin{aligned}
    |\mathbb{E} h(\widetilde{X}) - \mathbb{E} h(\widetilde{Y})|
    &\leq \frac{1}{N} \sum_{j=1}^N \sum_{k=1}^q \ev{ \left| g_{\beta}^{(k)}(S(\widetilde{Y})) - x^{(k)} \right| \cdot \left| f_h(\widetilde{Y}^{(j,k)}) - f_h(\widetilde{Y}) \right| \indicator{\widetilde{Y}^{(j, k)} \in \widetilde{\Omega}} }.
\end{aligned}
\end{equation}
Here, we may extend the definition of $f_h$ outside its domain $\widetilde{\Omega}(x, r)$, with a slight abuse of notation.
Next, we aim to bound $\left| f_h(\widetilde{Y}^{(j,k)}) - f_h(\widetilde{Y}) \right|$, uniformly over all the possible pairs of states $\widetilde{Y}, \widetilde{Y}^{(j, k)}$. By Lemma~\ref{lem:glauber_bound_fh_diff}, we have
\begin{equation} \label{b0}
    |f_h(\sigma) - f_h(\tau)| \leq \norm{L(h)}_\infty \sum_{t=1}^\infty \sum_{i=1}^N \mathbb{P}^{x}_{\sigma, \tau}\left( W_t(i) \ne Z_t(i) \right)
\end{equation}
for any $\sigma, \tau \in \widetilde{\Omega}(x, r)$, and any sequence of couplings $(W_t, Z_t)$ of the $t$-step distributions of $\widetilde{X}_t$, starting from $\sigma$ and $\tau$. Let $t_N := \gamma N \log(N)^2$, where $\gamma > 0$ is a constant to be chosen large enough later. Then we may decompose the sum in~\eqref{b0} as follows:
\begin{equation} \label{eq:split_sum}
    \sum_{t=0}^{\infty} \sum_{i=1}^N \mathbb{P}^{x}_{\sigma, \tau}\left( W_t(i) \ne Z_t(i) \right)
    = \sum_{t=0}^{t_N-1} \mathbb{E}^{x}_{\sigma, \tau}\left[ d_H\left( W_t, Z_t \right) \right]
    + \sum_{t=t_N}^{\infty} \sum_{i=1}^N \mathbb{P}^{x}_{\sigma, \tau}\left( W_t(i) \ne Z_t(i) \right).
\end{equation}

First, we will control the tail sum in~\eqref{eq:split_sum} by appealing to the mixing time of $\widetilde{X}_t$, which was shown to satisfy $t_{\mathrm{mix}}^x = O(N \log(N))$ in Theorem~\ref{mixingtime}. Suppose that for each $t \geq t_N$, we choose $(W_t, Z_t)$ to be an optimal coupling of the $t$-step distributions of $\widetilde{X}_t$, so that the probability that $W_t \ne Z_t$ is equal to the total variation distance between the corresponding $t$-step distributions. Since this decays geometrically for multiples of the mixing time by~\cite[Lemma~4.10 and Equation~(4.33)]{levin2017markov}, there exists a constant $c > 0$ such that for all $t \geq t_N$ and vertices $i \in [N]$,
\[
    \mathbb{P}^{x}_{\sigma, \tau}\left( W_t(i) \ne Z_t(i) \right)
    \leq 2^{-t / t^x_{\mathrm{mix}}}
    \leq \exp\left\{ -\frac{ct}{N \log(N)} \right\}.
\]
Therefore, by using the inequality $e^{-a} \leq 1 - a/2$ for $0 \leq a \leq 3 / 2$ to simplify the geometric series, and choosing $\gamma$ to be large enough (e.g.\ $\gamma > 3/c$), for large enough $N$, we have
\begin{equation} \label{b2}
    \sum_{t=t_N}^{\infty} \sum_{i=1}^N \mathbb{P}^{x}_{\sigma, \tau}\left( W_t(i) \ne Z_t(i) \right)
    \leq \sum_{t=t_N}^{\infty} N \exp\left\{ -\frac{ct}{N \log(N)} \right\}
    \leq \frac{2}{c} N^{2 - c \gamma} \log(N)
    \leq N^{-1}.
\end{equation}

Next, we will bound the finite sum in~\eqref{eq:split_sum}, which requires a more delicate analysis to show that $\widetilde{X}_t$ can be coupled such that it is contracting on a good set. For $t < t_N$, let $(W_t, Z_t)$ be the coupling of two copies of $\widetilde{X}_t$ described in Lemma~\ref{GDcontrcting}. In the proof of Lemma~\ref{colgoodset}, we showed in~\eqref{eq:colgoodset_pf1} that if $\sigma, \tau \in \widetilde{\Omega}(x, \frac{r}{5})$, then for all $t \leq t_N$, this coupling satisfies
\begin{align*}
    \mathbb{E}^{x}_{\sigma, \tau}\left[ d_H\left( W_t, Z_t \right) \right]
    \leq \left( 1 - \frac{1 - \theta(x, \beta, q)}{2N} \right)^{t} d_H(\sigma, \tau) + 4 N \exp\left\{ -\frac{cN}{\gamma \log(N)^2} \right\}.
\end{align*}
Therefore, for any arbitrary $\sigma \in \widetilde{\Omega}(x, r)$ and $\tau = \sigma^{(v, k)}$ (with $d_H(\sigma, \tau) = 1$), we have
\begin{align}
    \sum_{t=0}^{t_N-1} \mathbb{E}^{x}_{\sigma, \tau}\left[ d_H\left( W_t, Z_t \right) \right]
    &\leq \sum_{t=0}^{t_N - 1} \left( 1 - \frac{1 - \theta(x, \beta, q)}{2N} \right)^{t} \indicator{\sigma, \tau \in \widetilde{\Omega}(x, \frac{r}{5})}
    + 4 N e^{-\frac{cN}{\gamma\log(N)^2}} \indicator{\sigma, \tau \in \widetilde{\Omega}(x, \frac{r}{5})} \nonumber\\
    &\qquad+ N t_N \indicator{\{ \sigma \notin \widetilde{\Omega}(x, \frac{r}{5})\} \cup \{ \tau \notin \widetilde{\Omega}(x, \frac{r}{5})\}}. \label{b3}
\end{align}
Note that we may further bound the geometric series by $2N / (1 - \theta(x, \beta, q))$. Furthermore, since $\sigma$ and $\tau$ differ in at most one site, $\sigma \in \widetilde{\Omega}(x, \frac{r}{10})$ implies that both $\sigma, \tau \in \widetilde{\Omega}(x, \frac{r}{5})$. By combining the bounds~\eqref{eq:split_sum}--\eqref{b3} in~\eqref{b0}, we obtain
\[
    \left| f_h(\sigma^{(j,k)}) - f_h(\sigma) \right|
    \leq \norm{L(h)}_\infty N \left( \frac{2}{1 - \theta(x, \beta, q)} 
    + 4 \exp\left\{ -\frac{cN}{\gamma \log(N)^2} \right\}
    + N^{-1} + t_N \indicator{\sigma \notin \widetilde{\Omega}(x, \frac{r}{10})} \right).
\]
Plugging this bound back into~\eqref{steinbound} yields
\begin{align} \label{steinbound2}
    \lvert \mathbb{E}h(\widetilde{X}) - \mathbb{E}h(\widetilde{Y}) \rvert
    &\leq \norm{L(h)}_\infty \left( \frac{2}{1 - \theta(x, \beta, q)} + o(1) \right) N\, \mathbb{E} \norm{g_{\beta}(S(\widetilde{Y})) - x}_1
    \nonumber\\
    &\qquad + \norm{L(h)}_\infty q \gamma N^2 \log(N)^2 \, \mathbb{P}\left( \widetilde{Y} \notin \widetilde{\Omega}(x, \mbox{$\frac{r}{10}$}) \right).
\end{align}
It remains to show that the first term in~\eqref{steinbound2} is of order $O(\sqrt{N})$ and the second term is of order $o(1)$ by analysing the concentration of the random vector $\widetilde{Y}$.

Let $Y$ be the product measure on $\Omega(x, r)$ where the colour of each vertex is independently distributed according to the probability vector $x$. Observe that $N S(Y)$ is a multinomial random vector with $N$ trials and probabilities $\mathbb{E}\left[ S(Y) \right] = x$, and the variance of each component is given by $\mathrm{Var}\left( S^{(k)}(Y) \right) = N^{-1} x^{(k)} (1 - x^{(k)})$. By using concentration inequalities for the multinomial distribution (e.g.\ the Bretagnolle--Huber--Carol inequality~\cite[Proposition A.6.6]{vandervaart1996weak}), we have
\begin{equation} \label{eq:multinomial_conc1}
    \mathbb{P}\left( Y \notin \widetilde{\Omega}(x, \mbox{$\frac{r}{10}$}) \right) \leq O(e^{-c_1 N})
    \quad \text{and} \quad
    \mathbb{P}\left( Y \in \widetilde{\Omega}(x, r) \right) \geq 1 - O(e^{-c_2 N}) \geq \frac{1}{4},
\end{equation}
for some constants $c_1, c_2 > 0$, and $N$ large enough. Therefore, we may pass from the conditional distribution $\widetilde{Y}$ to the i.i.d.\ vector $Y$ by using~\eqref{eq:multinomial_conc1} to obtain
\begin{align} \label{eq:multinomial_conc2}
    \mathbb{P}\left( \widetilde{Y} \notin \widetilde{\Omega}(x, \mbox{$\frac{r}{10}$}) \right)
    = \mathbb{P}\left( Y \notin \widetilde{\Omega}(x, \mbox{$\frac{r}{10}$}) \mid Y \in \widetilde{\Omega}(x, r) \right)
    \leq \frac{\mathbb{P}\left( Y \notin \widetilde{\Omega}(x, \mbox{$\frac{r}{10}$}) \right)}{\mathbb{P}\left( Y \in \widetilde{\Omega}(x, r) \right)}
    \leq 4 e^{-c_1 N} .
\end{align}

Furthermore, since the Lipschitz constant of $g_\beta$ is bounded by $\theta(x, \beta, q) + O(r)$ from Lemma~\ref{L1contract}, for sufficiently small $r$, we have
\begin{equation} \label{pxbound1}
\begin{aligned} 
    \mathbb{E}\lVert g_{\beta}(S(\widetilde{Y})) - g_{\beta}\left( x \right) \rVert_1
    &\leq (\theta(x, \beta, q) + O(r)) \cdot \mathbb{E}\lVert S(\widetilde{Y}) - x \rVert_1 \\
    &\leq 2 \theta(x, \beta, q) \cdot \mathbb{E}\lVert S(\widetilde{Y}) - x \rVert_1.
\end{aligned}
\end{equation}
Since the $\ell_1$ norm in $\reals^q$ is bounded by $\sqrt{q}$ times the $\ell_2$ norm, applying Jensen's inequality and then using~\eqref{eq:multinomial_conc1} to pass from the conditional distribution $\widetilde{Y}$ to $Y$ shows that
\begin{equation} \label{pxbound2}
    \mathbb{E}\lVert S(\widetilde{Y}) - x \rVert_1
    \leq \sqrt{q} \sqrt{\sum_{k=1}^q \mathbb{E} \left( S^{(k)}(\widetilde{Y}) - x^{(k)} \right)^2}
    \leq 2 \sqrt{q} \sqrt{\sum_{k=1}^q \mathbb{E} \left( S^{(k)}(Y) - x^{(k)} \right)^2}
    \leq \frac{q}{\sqrt{N}}.
\end{equation}
To conclude, by combining \eqref{eq:multinomial_conc2}, \eqref{pxbound1}, and~\eqref{pxbound2} in~\eqref{steinbound2}, we have shown that
\[
    \lvert \mathbb{E}h(\widetilde{X}) - \mathbb{E}h(\widetilde{Y}) \rvert \leq \norm{L(h)}_\infty \left( \frac{4q \theta(x, \beta, q)}{1 - \theta(x, \beta, q)} + o(1) (1 + 2q\theta(x, \beta, q)) \right) \sqrt{N},
\]
which completes the proof.
\end{proof}

\begin{remark} \label{rmk:thetastar_limit}
We proved that the main result of Theorem~\ref{WDbound} holds with constant
\begin{equation}
    \theta^* \leq \frac{4q \theta(x, \beta, q)}{1 - \theta(x, \beta, q)} + o(1) (1 + 2q\theta(x, \beta, q))
\end{equation}
for any $x \in \mathfrak{S}_{\beta, q}$, where $o(1) \to 0$ as $N \to \infty$.
Note that if $\theta(x, \beta, q) \to 0$ and $N \to \infty$ simultaneously, then this implies that $\theta^* \to 0$. We will show that this indeed occurs in the low-temperature regime, as $\beta \to \infty$, and in the high-temperature regime, as $\beta \to 0$.
First, if $x = \mathbf{T}^j\check{s}_{\beta,q}$, $j \in [q]$, then recall that from Lemma~\ref{L1contract} and Lemma~\ref{expforcond}, we have
\[
    \theta(x, \beta, q) = a = 2\beta q s^*_{\beta,q} \frac{1 - s^*_{\beta,q}}{q-1} < \lambda(x, \beta, q).
\]
Using the fact that $s^*_{\beta,q} \to 1$ as $\beta \to \infty$. (Theorem~\ref{pointofconcentration}) and the expression~\eqref{eq:lambda_expr} for $\lambda(x, \beta, q)$ given in the appendix, it can be shown that $\lambda(x, \beta, q) \to 0$ as $\beta \to \infty$. Hence, we conclude that in the low-temperature regime, $\theta(x, \beta, q) \to 0$ as $\beta \to \infty$.
Similarly, in the high-temperature regime, if $x = \hat{e}$, then $\theta(x, \beta, q) = 2 \beta / q \to 0$ as $\beta \to 0$.
\end{remark}

\subsection{Approximation result in high-temperature regime} \label{pfWDBhigh}

We conclude by discussing the modifications of the proof in the previous section needed to prove our main result when $\beta < \beta_s$ and there is a unique equilibrium macrostate $x = \hat{e}$ in $\mathfrak{S}_{\beta, q}$.

\begin{proof}[Proof of Theorem~\ref{WDbound_high}]
When $\beta < \beta_s$, Theorem~\ref{WDbound_high} can be proved by considering the usual Glauber dynamics for $\mu$ and $\nu$ directly, instead of the restricted dynamics. The proof follows the same structure as the proof of Theorem~\ref{WDbound} in Section~\ref{pfWDB}, with the following changes required:
\begin{enumerate}[label=\normalfont{(\arabic*)}]
    \item Obtaining a similar bound as Theorem~\ref{mixingtime} for the mixing time of the Glauber dynamics for the Curie--Weiss--Potts model.
    \item Showing that the proportions chain of the Glauber dynamics concentrates around $\hat{e}$ in a ball of constant order (or smaller) for a sufficiently long period, analogous to Part~\ref{GlauberConcen_part1} of Lemma~\ref{GlauberConcen}.
    \item Proving that the Glauber dynamics is contracting as long as it is close enough to $\hat{e}$, as described by~\eqref{glaubercontract} in Lemma~\ref{GDcontrcting}. 
\end{enumerate}
The first two items have already been proved in~\cite{cuff2012}: \cite[Theorem~1]{cuff2012} shows that the mixing time of the Glauber dynamics is of order $O(N \log(N))$, and~\cite[Proposition~3.3, Part~(1)]{cuff2012} is a direct analogue of Part~\ref{GlauberConcen_part1} of Lemma~\ref{GlauberConcen}. Finally, since the two chains in the proof of Lemma~\ref{GDcontrcting} are separated from the boundary in the analysis, it is apparent that~\eqref{glaubercontract} also holds for the same coupling of the Glauber dynamics without the rejection step.
\end{proof}

In the temperature regime $\beta_s \leq \beta < \beta_c$, where there is still a unique equilibrium macrostate $x = \hat{e}$ in $\mathfrak{S}_{\beta,q}$, we may want to approximate the unconditional Curie--Weiss--Potts model by a sequence of i.i.d.\ uniform spins.
In this setting, the Glauber dynamics mixes rapidly after excluding a subset of initial configurations (that are far from $\hat{e}$) with a probability mass exponentially small in $N$~\cite[Theorem~4]{cuff2012}. However, since the (worst-case) mixing time is exponentially large in $N$~\cite[Theorem~3]{cuff2012}, the proof of Theorem~\ref{WDbound_high} described above does not work (in particular, the approach to control the tail sum in~\eqref{eq:split_sum} fails).
Since the Curie--Weiss--Potts model still concentrates in $\widetilde{\Omega}(\hat{e}, r)$, we are able to prove the following weaker result by using Theorem~\ref{WDbound} and the definition of conditional expectation:

\begin{proposition} \label{WDbound_med}
Suppose that $\beta_s \leq \beta < \beta_c$. Let $X \in \Omega$ be distributed according to the Curie--Weiss--Potts model and $Y \in \Omega$ be a random configuration with i.i.d.\ uniform spins. Then there exist constants $\theta^* > 0$ and $c, C \geq 0$ such that for any function $h: \Omega \to \reals$,
\begin{equation*}
    \left\lvert \mathbb{E} h(X) - \mathbb{E} h(Y) \right\rvert \leq \norm{L(h)}_\infty \theta^* \sqrt{N} + C \norm{h}_\infty e^{-cN} .
\end{equation*}
\end{proposition}

\begin{proof}
Let $\widetilde{X}$ and $\widetilde{Y}$ be the conditional random configurations in $\widetilde{\Omega}(\hat{e}, r)$ as defined in Theorem~\ref{WDbound} with $r$ chosen to be sufficiently small. Note that $\widetilde{X} \stackrel{\mathrm{d}}{=} X \mid X \in \widetilde{\Omega}(\hat{e}, r)$ and $\widetilde{Y} \stackrel{\mathrm{d}}{=} Y \mid Y \in \widetilde{\Omega}(\hat{e}, r)$. Observe that we can write $\E h(X)$ as
\[
    \mathbb{E}\left[ h(X) \mid X \in \widetilde{\Omega}(\hat{e}, r) \right] + \Prob{X \notin \widetilde{\Omega}(\hat{e}, r)} \left( \mathbb{E}\left[ h(X) \mid X \notin \widetilde{\Omega}(\hat{e}, r) \right] - \mathbb{E}\left[ h(X) \mid X \in \widetilde{\Omega}(\hat{e}, r) \right] \right),
\]
and we have a similar expression for $\E h(Y)$. By taking the difference between these two expressions and bounding $h$ uniformly by $\norm{h}_\infty$, we obtain
\begin{equation*}
    \left\lvert \mathbb{E} h(X) - \mathbb{E} h(Y) \right\rvert 
    \leq \left\lvert \mathbb{E} h(\widetilde{X}) - \mathbb{E} h(\widetilde{Y}) \right\rvert + 2 \norm{h}_\infty \left( \mathbb{P}\left( X \notin \widetilde{\Omega}(\hat{e}, r) \right) + \mathbb{P}\left( Y \notin \widetilde{\Omega}(\hat{e}, r) \right) \right).
\end{equation*}
By using large deviations results for the Curie--Weiss--Potts model (see, e.g.,~\cite[Section~2.2]{cuff2012}) and the multinomial distribution~\eqref{eq:multinomial_conc1}, we deduce that $X$ and $Y$ are not in $\widetilde{\Omega}(\hat{e}, r)$ with exponentially small probability. Applying Theorem~\ref{WDbound} to the first term completes the proof.
\end{proof}


\begin{appendix}

\section{Deferred proofs} \label{app:deferred_proofs}

We collect the proofs of Lemmas~\ref{jacobian}, \ref{L1contract}, and~\ref{expforcond} for the Curie--Weiss--Potts model in this section. Recall from~\eqref{updatingprobvector} that the vector $g_\beta(s) \in \mathcal{S}$ has entries
\[
    g_{\beta}^{(k)}(s) = \frac{e^{2\beta s^{(k)}}}{\sum_{j=1}^{q} e^{2\beta s^{(j)}}}
    \quad\text{for}\quad
    k \in [q],\, s \in \reals^q.
\]
Furthermore, recall from~\eqref{s^*} and~\eqref{T1} that
\[
    s^*_{\beta,q} = \frac{1 + (q-1)s_{\beta,q}}{q}
    \quad\text{and}\quad
    \check{s}_{\beta,q} = \left( s^*_{\beta,q}, \frac{1 - s^*_{\beta,q}}{q-1}, \ldots, \frac{1 - s^*_{\beta,q}}{q-1} \right) \in \mathcal{S},
\]
where $s_{\beta,q}$ is the largest solution of the equation~\eqref{pointofconcentration_eq1} in Theorem~\ref{pointofconcentration}. Finally, recall the definitions of $a, a', b > 0$ from Lemma~\ref{jacobian}:
\begin{equation*}
    a = 2\beta q s^*_{\beta,q} \frac{1 - s^*_{\beta,q}}{q-1}, \quad
    a' = 2\beta \frac{1 - s^*_{\beta,q}}{q-1}, \quad
    b = 2\beta \frac{1 - s^*_{\beta,q}}{q - 1} \left( s^*_{\beta,q} - \frac{1 - s^*_{\beta,q}}{q - 1}\right).
\end{equation*}
Observe that these constants are related by the identity
\begin{equation} \label{aa'b_identity}
    a - a' = 2\beta (q s^*_{\beta,q} - 1) \frac{1 - s^*_{\beta,q}}{q-1} = (q-1)b.
\end{equation}

\subsection{Proof of Lemma~\ref{jacobian}} \label{gradient_proof}

Let $s_1, s_2 \in \mathcal{S}$. Since $s_1, s_2$ are both in the probability simplex $\mathcal{S}$, we have $\sum_{k=1}^q (s_1^{(k)} - s_2^{(k)}) = 0$. It can also be verified that
\begin{align*}
    \frac{\partial}{\partial s^{(j)}}g^{(k)}_{\beta}(s)=
    \begin{cases}
        -2\beta g^{(j)}_{\beta}g^{(k)}_{\beta}, \quad &k \neq j,\\
        -2\beta(g^{(k)}_{\beta})^2+2\beta g^{(k)}_{\beta}, \quad &k = j.
    \end{cases}
\end{align*}
First, we consider the case $x = \hat{e}$. Since $g_\beta(\hat{e}) = \hat{e}$ from Lemma~\ref{g(x)=x}, for all $k = 1, \dots, q$, we have
\begin{align*}
    \nabla g^{(k)}_{\beta}(\hat{e})(s_1-s_2)^{\top}
    &=\left(-\frac{2\beta}{q^2}+\frac{2\beta}{q}\right)(s_1^{(k)}-s_2^{(k)})+\sum_{j\neq k} \left(-\frac{2\beta}{q^2}\right)(s_1^{(j)}-s_2^{(j)})\\
    &=\frac{2\beta}{q}(s_1^{(k)}-s_2^{(k)})+\left(-\frac{2\beta}{q^2}\right)\sum_{j=1}^q (s_1^{(j)}-s_2^{(j)})
    =\frac{2\beta}{q}(s_1^{(k)}-s_2^{(k)}).
\end{align*}

Next, we will consider the case $x = \mathbf{T}^1\check{s}_{\beta,q} = \check{s}_{\beta,q}$. By symmetry, the proof for the other $\mathbf{T}^j\check{s}_{\beta,q}$, $j = 2, \dots, q$, is identical. Again, by Lemma~\ref{g(x)=x}, we have $g_{\beta}(\check{s}_{\beta,q}) = \check{s}_{\beta,q}$. Since the first coordinate of $\check{s}_{\beta,q}$ differs from the rest, we will consider the first coordinate separately:
\begin{align*}
    \nabla g^{(1)}_{\beta}(\check{s}_{\beta,q})(s_1-s_2)^{\top}
    &=\left(-2\beta (s^*_{\beta,q})^2+2\beta s^*_{\beta,q}\right)(s_1^{(1)}-s_2^{(1)})+\sum_{k=2}^q \left(-2\beta s^*_{\beta,q} \frac{1-s^*_{\beta,q}}{q-1}\right)(s_1^{(k)}-s_2^{(k)})\\
    &=2\beta s^*_{\beta,q}\left(1-s^*_{\beta,q}+\frac{1-s^*_{\beta,q}}{q-1}\right)(s_1^{(1)}-s_2^{(1)})-\left(2\beta s^*_{\beta,q} \frac{1-s^*_{\beta,q}}{q-1}\right)\sum_{k=1}^q (s_1^{(k)}-s_2^{(k)})\\
    &=2\beta s^*_{\beta,q}\left(1-s^*_{\beta,q}+\frac{1-s^*_{\beta,q}}{q-1}\right)(s_1^{(1)}-s_2^{(1)}).
\end{align*}
For the remaining coordinates $k = 2, \dots, q$, it suffices to consider $k = 2$ by symmetry. We have
\begin{align*}
    \nabla g^{(2)}_{\beta}(\check{s}_{\beta,q})(s_1-s_2)^{\top}
    &=\left(-2\beta s^*_{\beta,q} \frac{1-s^*_{\beta,q}}{q-1}\right)(s_1^{(1)}-s_2^{(1)})+\left(-2\beta \left(\frac{1-s^*_{\beta,q}}{q-1}\right)^2+2\beta \frac{1-s^*_{\beta,q}}{q-1}\right)(s_1^{(2)}-s_2^{(2)})\\
    &\qquad\qquad\qquad\qquad\qquad\qquad\qquad\qquad+\sum_{k=3}^q \left(-2\beta \left(\frac{1-s^*_{\beta,q}}{q-1}\right)^2\right)(s_1^{(k)}-s_2^{(k)})\\
    &=2\beta \frac{1-s^*_{\beta,q}}{q-1}(s_1^{(2)}-s_2^{(2)})-2\beta\frac{1-s^*_{\beta,q}}{q-1}\left(s^*_{\beta,q}(s_1^{(1)}-s_2^{(1)})+\sum_{k=2}^q \frac{1-s^*_{\beta,q}}{q-1}(s_1^{(k)}-s_2^{(k)})\right)\\
    &=2\beta \frac{1-s^*_{\beta,q}}{q-1}\left((s_1^{(2)}-s_2^{(2)})+\left(-s^*_{\beta,q}+\frac{1-s^*_{\beta,q}}{q-1}\right)(s_1^{(1)}-s_2^{(1)})\right).
\end{align*}
Writing the above displayed expressions in matrix form completes the proof. \hfill$\square$

\subsection{Proof of Lemma~\ref{L1contract}} \label{L1contract_proof}

Let $x \in \mathfrak{S}_{\beta,q}$. Observe that $g_\beta$ is a smooth function (in particular, it has continuous and bounded second partial derivatives on the compact probability simplex $\mathcal{S}$). We claim that for any $s_1, s_2 \in \mathcal{S}$,
\begin{equation} \label{LipCon}
    \norm{\mathbf{J}(x) (s_1 - s_2)^{\top}}_1 \leq \theta(x, \beta, q) \norm{s_1 - s_2}_1,
\end{equation}
where $\mathbf{J}(x)$ denotes the Jacobian matrix of $g_{\beta}$ at $x$ (see Lemma~\ref{jacobian}).
Assuming that~\eqref{LipCon} holds, then by using a Taylor series expansion of $g_\beta^{(k)}$ around $s_1$ for any $k \in [q]$, we obtain
\[
    g^{(k)}_{\beta}(s_1) - g^{(k)}_{\beta}(s_2) = \nabla g^{(k)}_\beta(s_1) (s_1 - s_2)^{\top} + O(\lVert s_1 - s_2 \rVert_1^2).
\]
By assumption, $\norm{s_1 - s_2}_1 \leq \norm{s_1 - x}_2 + \norm{s_2 - x}_2 \leq 2r$. By summing up coordinates and writing $\mathbf{J}(s_1)$ to denote the Jacobian matrix of $g_\beta$ at $s_1$, this implies that
\[
    \norm{g_{\beta}(s_1) - g_{\beta}(s_2))}_1 \leq \norm{\mathbf{J}(s_1) (s_1 - s_2)^{\top}}_1 + O(r) \norm{s_1 - s_2}_1.
\]
Furthermore, by the triangle inequality, we have
\[
    \norm{\mathbf{J}(s_1) (s_1 - s_2)^{\top}}_1 \leq \norm{\mathbf{J}(x) (s_1 - s_2)^{\top}}_1 + \norm{(\mathbf{J}(s_1) - \mathbf{J}(x)) (s_1 - s_2)^{\top}}_1.
\]
By~\eqref{LipCon}, the first term is bounded by $\theta(x, \beta, q) \norm{s_1 - s_2}_1$. Due to the smoothness of $g_{\beta}$ (i.e.\ using its higher derivatives, which are bounded on the compact probability simplex), the second term can be bounded by $O(r) \norm{s_1 - s_2}_1$. This implies that the desired claim holds:
\[
    \norm{g_\beta(s_1) - g_\beta(s_2)}_1 \leq (\theta(x, \beta, q) + O(r)) \norm{s_1 - s_2}_1.
\]

We will now prove that~\eqref{LipCon} holds. Let $\mathbf{A}(x)$ denote the matrix related to the Jacobian of $g_\beta$ at $x$ from Lemma~\ref{jacobian} such that $\mathbf{J}(x)(s_1 - s_2)^{\top} = \mathbf{A}(x) (s_1 - s_2)^{\top}$ for any $s_1, s_2 \in \mathcal{S}$. If $x = \hat{e}$, then we simply have
\begin{equation*}
    \lVert \mathbf{J}(\hat{e}) (s_1 - s_2)^{\top} \rVert_1
    = \frac{2 \beta}{q} \lVert s_1 - s_2 \rVert_1
    = \theta(\hat{e}, \beta, q) \lVert s_1 - s_2 \rVert_1.
\end{equation*}
Next, we consider $x = \check{s}_{\beta,q}\equiv\mathbf{T}^1\check{s}_{\beta,q}$. By symmetry, the proof is identical for the other $\mathbf{T}^j\check{s}_{\beta,q}$, $j = 2, \dots, q$. Since $s_1, s_2 \in \mathcal{S}$, $(s_1^{(1)} - s_2^{(1)}) = -\sum_{k=2}^q (s_1^{(k)} - s_2^{(k)})$, and so by the triangle inequality,
\begin{equation} \label{domcolbdd}
    \lvert s_1^{(1)} - s_2^{(1)} \rvert \leq \sum_{k=2}^q \lvert s_1^{(k)} - s_2^{(k)} \rvert.
\end{equation}
By putting in the form of $\mathbf{A}(\check{s}_{\beta,q})$, and then using the triangle inequality and~\eqref{domcolbdd}, we obtain
\begin{align}
    \lVert\mathbf{J}(\check{s}_{\beta,q})(s_1 - s_2)^{\top} \rVert_1
    &= a \lvert s_1^{(1)} - s_2^{(1)} \rvert + \sum_{k=2}^q \lvert a'(s_1^{(k)} - s_2^{(k)}) - b(s_1^{(1)} - s_2^{(1)}) \rvert \nonumber\\
    &\leq a\lvert s_1^{(1)} - s_2^{(1)} \rvert + \sum_{k=2}^q (a' + (q - 1) b) \lvert s_1^{(k)} - s_2^{(k)} \rvert. \label{asymjacob}
\end{align}
Since we have the identity $a' + (q-1)b = a$ from~\eqref{aa'b_identity}, this shows that~\eqref{LipCon} holds:
\begin{equation*}
    \lVert \mathbf{J}(\check{s}_{\beta,q})(s_1 - s_2)^{\top} \rVert_1 \leq a \lVert s_1 - s_2 \rVert_1 = \theta(\check{s}_{\beta,q}, \beta, q) \lVert s_1 - s_2 \rVert_1.
\end{equation*}
This completes the proof. \hfill$\square$

\subsection{Proof of Lemma~\ref{expforcond}} \label{expforcond_proof}

Let $\mathbf{A}(x)$ be the matrix related to the Jacobian of $g_\beta$ at $x$ defined in Lemma~\ref{jacobian}. For Part~\ref{expforcond_part1}, when $x = \hat{e}$ and $\mathbf{A}(\hat{e}) = (2 \beta / q) \mathbf{I}$, we have $\lambda(\hat{e}, \beta, q) = 2\beta / q < 1$ for $\beta \leq \beta_c$, since $\beta_c < q/2$ from Theorem~\ref{pointofconcentration}.

For Part~\ref{expforcond_part2}, when $x = \mathbf{T}^j\check{s}_{\beta,q}$, we can use Theorem~\ref{pointofconcentration} again to deduce that $0 < b < a'$. Since $s_{\beta, q}$ is strictly increasing on $[\beta_c, \infty)$, we have $1 > s^*_{\beta,q} > (1 - s^*_{\beta,q})/(q-1) > 0$ for $\beta \geq \beta_c$, which implies that $0 < b < a'$. Furthermore, $a' < a$ follows from the identity $a = a' + (q-1)b$ from~\eqref{aa'b_identity}.
Thus, it remains to show that $a < \lambda(x, \beta, q) < 1$ for $\beta \geq \beta_c$.  By symmetry, it suffices to consider the case $x = \mathbf{T}^1\check{s}_{\beta,q} = \check{s}_{\beta,q}$. Due to the special form of $(\mathbf{A} + \mathbf{A}^{\top}) / 2$, its eigenvalues can be explicitly computed: it has a repeated eigenvalue $\lambda_i := a'$, $i = 2, \ldots, q-1$ with multiplicity $q - 2$, and its remaining two eigenvalues are given by
\begin{align*}
    &\lambda_1 := \frac{1}{2} \left( a + a' + \sqrt{(a - a')^2 + (q - 1) b^2} \right),\\
    &\lambda_q := \frac{1}{2} \left( a + a' - \sqrt{(a - a')^2 + (q - 1) b^2} \right).
\end{align*}
Recall the identity $a - a' = (q - 1)b$ from~\eqref{aa'b_identity}. Thus, we may write $\lambda_1$ and $\lambda_q$ as
\begin{equation*}
    \lambda_1 = \frac{1}{2} \left( a + a' + (a - a') \sqrt{1 + \frac{(q - 1)b^2}{(a - a')^2}}\right)
    = \frac{1}{2} \left( a + a' + (a - a') \sqrt{1 + \frac{1}{q - 1}} \right),
\end{equation*}
and 
\begin{equation*}
    \lambda_q
    = \frac{1}{2} \left( a + a' + (a - a') \sqrt{1 - \frac{(q - 1)b^2}{(a - a')^2}} \right)
    = \frac{1}{2} \left( a + a' - (a - a') \sqrt{1 + \frac{1}{q - 1}} \right).
\end{equation*}
By further simplifying, we may write
\begin{equation*}
    \lambda_1 = a + \frac{1}{2}\left( \sqrt{\frac{q}{q-1}} - 1 \right) (a - a'),
    \quad\text{and}\quad
    \lambda_q = a' - \frac{1}{2} \left( \sqrt{\frac{q}{q-1}} - 1 \right) (a - a').
\end{equation*}
Since we know that $a - a' > 0$, we deduce that $\lambda_1 > \lambda_2 = \cdots = \lambda_{q-1} > \lambda_q$ and $a < \lambda _1$. (Note that $a \leq \lambda_1$ holds for all $\beta > 0$.) By further plugging in the expressions for $a$ and $a'$, we obtain
\begin{align}
    &\lambda_1 = 2\beta \frac{1 - s^*_{\beta,q}}{q - 1} \left( (q s^*_{\beta,q} + 1) + \sqrt{\frac{q}{q-1}} (q s^*_{\beta,q} - 1) \right), \label{jacobian_lambda1_expr1} \\
    &\lambda_q = 2\beta \frac{1 - s^*_{\beta,q}}{q - 1} \left( (q s^*_{\beta,q} + 1) - \sqrt{\frac{q}{q-1}} (q s^*_{\beta,q} - 1) \right). \label{jacobian_lambdaq_expr2}
\end{align}
Since $2\beta (1 - s^*_{\beta,q}) / (q-1) > 0$ and
\begin{equation*}
    (qs^*_{\beta,q}+1)-\sqrt{\frac{q}{q-1}}(qs^*_{\beta,q}-1) \geq (q+1)-\sqrt{\frac{q}{q-1}}(q-1) > 0, \quad \text{for all} \quad q \geq 3,
\end{equation*}
it follows from~\eqref{jacobian_lambdaq_expr2} that $\lambda_q > 0$. Hence, all the eigenvalues of $(\mathbf{A} + \mathbf{A}^{\top}) / 2$ are positive, and its maximum absolute eigenvalue is given by $\lambda(x, \beta, q) = \lambda_1$. It remains to show that $\lambda_1 < 1$. First, by Lemma~\ref{g(x)=x} (or a simple substitution), $s^*_{\beta,q}$ solves the equation
\begin{equation*}
    s = g^1_{\beta}\left( \left( s, \frac{1-s}{q-1}, \ldots, \frac{1-s}{q-1} \right) \right)
    = \frac{e^{2\beta s}}{e^{2\beta s} + (q-1)e^{2\beta\frac{1-s}{q-1}}}.
\end{equation*}
By rearranging this equation, we deduce that $s^*_{\beta, q}$ satisfies
\begin{equation*}
    \exp\left\{ -2\beta \frac{qs^*_{\beta,q} - 1}{q - 1} \right\} = \frac{1 - s^*_{\beta,q}}{(q-1) s^*_{\beta,q}}.
\end{equation*}
By using this identity in~\eqref{jacobian_lambda1_expr1}, the maximum absolute eigenvalue $\lambda_1$ can be written as
\begin{equation} \label{eq:lambda_expr}
    \lambda_1 = \frac{1 - s^*_{\beta,q}}{2(qs^*_{\beta,q} - 1)} \left( (q s^*_{\beta,q} + 1) + \sqrt{\frac{q}{q-1}} (qs^*_{\beta,q} - 1) \right) \log\left( \frac{(q - 1) s^*_{\beta,q}}{1 - s^*_{\beta,q}} \right),
\end{equation}
and therefore 
\begin{equation*}
    \lambda_1 - 1 = \frac{1 - s^*_{\beta,q}}{2(qs^*_{\beta,q} - 1)} \left( (q s^*_{\beta,q} + 1) + \sqrt{\frac{q}{q-1}} (qs^*_{\beta,q} - 1) \right)f(s^*_{\beta,q}),
\end{equation*}
where $f$ is the function defined by
\begin{equation*}
    f(s) = \log\left( \frac{(q-1) s}{1 - s} \right) - \frac{2(qs - 1)}{(1 - s) \left( (qs + 1) + \sqrt{\frac{q}{q-1}}(qs - 1) \right)}.
\end{equation*}
Since $s^*_{\beta, q} \in [1 - 1/q, 1]$, to show that $\lambda_1 - 1 < 0$, it suffices to show that for all $q \geq 3$, the function $f$ is decreasing for $s \in [1 - 1/q, 1]$ and $f(1 - 1/q) < 0$. This may be verified through an exercise in calculus as follows. First, define
\begin{equation*}
    g(s) = \log\left( \frac{(q-1)s}{1-s} \right) - \frac{2(qs-1)}{(1-s) \left( (qs+1) + (5/4)(qs-1) \right)}.
\end{equation*}
Note that for all $q \geq 3$ and $s \geq 1 - 1/q$, we have $f(s) \leq g(s)$. Hence, it suffices to show that $g(1 - 1/q) < 0$ and $g'(s) < 0$ for $s \in [1 - 1/q, 1]$. It may be shown that
\[
    g(1 - 1/q) = 2 \log\left( q - 1 \right) - \frac{8 (q - 2)q}{9q - 10} < 0
\]
for all $q \geq 3$. Moreover,
\[
    g'(s) = \frac{-153 q^2 s^3 + 81 q(q+2) s^2 - (82q + 9)s + 1}{(1 - s)^2 s (1 - 9qs)^2}.
\]
Since the denominator is positive, this reduces to the task of showing that the numerator
\[
    h(s) = -153 q^2 s^3 + 81 q(q+2) s^2 - (82q + 9)s + 1
\]
is negative for all $q \geq 3$ and $s \in [1 - 1/q, 1]$. Indeed, since this is a cubic polynomial in $s$, this is much simpler to verify and we omit the remaining details.
\hfill$\square$

\end{appendix}

\section*{Acknowledgements}

The research of R.H.\ is supported by the Melbourne Research Scholarship.
We gratefully acknowledge Nathan Ross for facilitating our collaboration, many enlightening discussions, and providing valuable suggestions and feedback on our drafts. J.L.\ would also like to acknowledge the University of Melbourne, where part of this work was initially completed.
We would like to thank the anonymous referees for their helpful comments and suggestions.


\printbibliography

\end{document}